\documentclass[12pt]{amsart}
\usepackage[utf8]{inputenc}

\title{Clique packings in random graphs}
\author{SIMON GRIFFITHS \and LETÍCIA MATTOS}

\thanks{\rule[-.2\baselineskip]{0pt}{\baselineskip}
This project suggested to us by Robert Morris, and the authors would like to thank him for introducing us to the problem and suggesting the use of the method of differential equation to potentially solve it.\phantom{and}\hspace{10mm}
	SG was partially supported by FAPERJ (Proc. 201.194/2022) and by CNPq (Proc. 307521/2019-2), and 
	LM was partially supported by CAPES (Proc. 88887.807279/2023-00).
}

\address{Simon Griffiths \newline
	Departamento de Matemática, PUC-Rio, Rua Marquês de São Vicente 225, Gávea, 22451-900 Rio de Janeiro, Brasil \newline
	E-mail: \texttt{E-mail:simon@puc-rio.br}}

\address{Letícia Mattos \newline
	Institut für Informatik, Universität Heidelberg, Im Neuenheimer Feld 205, 69120 Heidelberg, Germany. \texttt{E-mail:mattos@uni-heidelberg.de}}

\usepackage[utf8]{inputenc}
\usepackage{ mathrsfs }
\usepackage{ dsfont }
\usepackage{amsmath}
\usepackage{indentfirst}
\usepackage{mathtools}
\usepackage{amsfonts}
\usepackage{amsthm}
\usepackage{graphicx}
\usepackage[english]{babel}
\usepackage{fullpage}
\usepackage{enumerate}
\usepackage{tasks}
\usepackage{mleftright}
\usepackage{thmtools}
\usepackage{thm-restate}

\usepackage{ amssymb }
\usepackage{ textcomp }

\newtheorem{theorem}{Theorem}[section]
\newtheorem{lemma}[theorem]{Lemma}
\newtheorem{prop}[theorem]{Proposition}

\newtheorem{claim}[theorem]{Claim}

\newtheorem{conjecture}[theorem]{Conjecture}

\theoremstyle{definition}
\newtheorem{defn}{Definition}

\theoremstyle{remark}

\def\A{\mathcal{A}}

\def\C{\mathcal{C}}

\def\F{\mathcal{F}}

\def\K{\mathcal{K}}

\def\NN{\mathcal{N}}

\def\S{\mathcal{S}}

\def\R{\mathbb{R}}
\def\Pr{\mathbb{P}}

\def\tq{\widetilde{Q}}
\def\ty{\widetilde{Y}}

\newcommand{\eps}{\varepsilon}
\newcommand{\Ex}[1]{\mathbb{E}\left[#1\right]}
\newcommand{\Exm}[1]{\mathbb{E}_{m_+}\left[#1\right]}
\newcommand{\pr}[1]{\mathbb{P}\left(#1\right)}

\newcommand{\se}{\subseteq}

\usepackage{hyperref}

\usepackage{cleveref}

\newcommand{\eq}[1]{\begin{equation}\label{eq:#1}}
	\newcommand{\eqe}{\end{equation}}
\newcommand{\eqr}[1]{\eqref{eq:#1}}

\usepackage{pagecolor}
\makeatletter
\let\save@mathaccent\mathaccent
\newcommand*\if@single[3]{%
	\setbox0\hbox{${\mathaccent"0362{#1}}^H$}%
	\setbox2\hbox{${\mathaccent"0362{\kern0pt#1}}^H$}%
	\ifdim\ht0=\ht2 #3\else #2\fi
}
\newcommand*\rel@kern[1]{\kern#1\dimexpr\macc@kerna}
\newcommand*\widebar[1]{\@ifnextchar^{{\wide@bar{#1}{0}}}{\wide@bar{#1}{1}}}
\newcommand*\wide@bar[2]{\if@single{#1}{\wide@bar@{#1}{#2}{1}}{\wide@bar@{#1}{#2}{2}}}
\newcommand*\wide@bar@[3]{%
	\begingroup
	\def\mathaccent##1##2{%
		\let\mathaccent\save@mathaccent
		\if#32 \let\macc@nucleus\first@char \fi
		\setbox\z@\hbox{$\macc@style{\macc@nucleus}_{}$}%
		\setbox\tw@\hbox{$\macc@style{\macc@nucleus}{}_{}$}%
		\dimen@\wd\tw@
		\advance\dimen@-\wd\z@
		\divide\dimen@ 3
		\@tempdima\wd\tw@
		\advance\@tempdima-\scriptspace
		\divide\@tempdima 10
		\advance\dimen@-\@tempdima
		\ifdim\dimen@>\z@ \dimen@0pt\fi
		\rel@kern{0.6}\kern-\dimen@
		\if#31
		\overline{\rel@kern{-0.6}\kern\dimen@\macc@nucleus\rel@kern{0.4}\kern\dimen@}%
		\advance\dimen@0.4\dimexpr\macc@kerna
		\let\final@kern#2%
		\ifdim\dimen@<\z@ \let\final@kern1\fi
		\if\final@kern1 \kern-\dimen@\fi
		\else
		\overline{\rel@kern{-0.6}\kern\dimen@#1}%
		\fi
	}%
	\macc@depth\@ne
	\let\math@bgroup\@empty \let\math@egroup\macc@set@skewchar
	\mathsurround\z@ \frozen@everymath{\mathgroup\macc@group\relax}%
	\macc@set@skewchar\relax
	\let\mathaccentV\macc@nested@a
	\if#31
	\macc@nested@a\relax111{#1}%
	\else
	\def\gobble@till@marker##1\endmarker{}%
	\futurelet\first@char\gobble@till@marker#1\endmarker
	\ifcat\noexpand\first@char A\else
	\def\first@char{}%
	\fi
	\macc@nested@a\relax111{\first@char}%
	\fi
	\endgroup
}
\makeatother

\begin{document}
	
	\begin{abstract}
		
		We consider the question of how many edge-disjoint near-maximal cliques may be found in the dense Erd\H os-Rényi random graph $G(n,p)$.  Recently Acan and Kahn showed that the largest such family contains only $O(n^2/(\log{n})^3)$ cliques, with high probability, which disproved a conjecture of Alon and Spencer.  We prove the corresponding lower bound, $\Omega(n^2/(\log{n})^3)$, by considering a random graph process which sequentially selects and deletes near-maximal cliques.  To analyse this process we use the Differential Equation Method.  We also give a new proof of the upper bound $O(n^2/(\log{n})^3)$ and discuss the problem of the precise size of the largest such clique packing. 
	\end{abstract}
	
	\maketitle

	\section{Introduction}
	
	Let $G(n,p)$ be the Erd\H{o}s--R\'enyi random graph on the vertex set $[n]:=\{1,2,\ldots,n\}$.
	Our interest lies in the size of the maximum $k$-clique packing in $G(n,p)$, which is simply a collection of edge-disjoint cliques.
	Frankl and R\"odl~\cite{FR85} were the pioneers in the study of packings in random graphs.
	In 1985, they showed that if $k$ is constant and $p \gg n^{\eps-2/(k+1)}$, then there is a \emph{nearly perfect} $k$-clique packing in $G(n,p)$ with high probability.
	That is, there exists a collection of edge-disjoint $k$-cliques covering almost all edges in $G(n,p)$.
    Up to logarithmic factors, $n^{-2/(k+1)}$ is the threshold for which every edge is contained in a $k$-clique. If $p \ll n^{-2/(k+1)}$, then this no longer holds, and hence we cannot have a near perfect $k$-clique packing in $G(n,p)$.
	
	In order to delve into what is known for non-constant values of $k$, we need some notation.
	Let  $E_p(n,k)$ denote the expected number of $k$-cliques in $G(n,p)$.
	That is,
	\[
	E_p(n,k)\, :=\, \binom{n}{k}p^{\binom{k}{2}}\, .
	\] 
	We define $k_0=k_0(n,p)$ to be the least integer for which the expected number of $k$-cliques in $G(n,p)$ is less than $1$, that is, $E_p(n,k_0)<1$.  In fact, $k_0$ is of the form
	\[
	k_0\, =\, (2+o(1)) \log_{1/p}{n}\, .
 \]
In the 1970s Matula~\cite{Matula70,Matula72} and, independently, Bollobás and Erd\H{o}s~\cite{BE} showed that the largest clique in $G(n,p)$ has either $k_0$ or $k_0-1$ vertices with high probability.  Therefore a clique of size $k_0-C$ for $C$ a constant, may be considered a near-maximal clique.  Incidentally, they also gave a more precise asymptotic expression for $k_0$.

	In the late 1980s, Bollobás~\cite{B88} showed that if $p \in (0,1)$ is constant and $k=k_0-4$, then $G(n,p)$ contains a $k$-clique packing of size $\Omega(n^2/k^4)$ with high probability.
	Given Bollob\'as' result it is natural to ask: how large can a packing of near-maximal cliques be?  
 
 In the book \emph{The Probabilistic Method}~\cite{AS}, by Alon and Spencer, it was conjectured that if $p \in (0,1)$ is constant and $k=k_0-4$, then $G(n,p)$ has $\Omega(n^2/k^2)$ edge-disjoint $k$-cliques with high probability.
	Note that this would be best possible, as we always have $\binom{n}{2}/\binom{k}{2}$ which is $O(n^2/k^2)$ as a trivial upper bound.
	Acan and Kahn~\cite{AK} recently disproved this conjecture by showing that for every constant $C\in \mathbb{N}_{\ge 0}$, setting $k=k_0-C$ we have that every $k$-clique packing in $G(n,p)$ has size at most $O(n^2/k^3)$ with high probability.
	Our main result gives a lower bound of the same order of magnitude.

	\begin{theorem}\label{thm:C} Let $p\in (0,1)$ and $C \in \mathbb{N}_{\ge 4}$ be constants and let $k=k_0-C$.  Then, with high probability, the random graph $G(n,p)$ contains at least 
		\[
		\frac{pn^2\log{n}}{40k^4}
		\]
		edge-disjoint $k$-cliques.
	\end{theorem}

 The constant given by Theorem~\ref{thm:C} is not best possible.  Finding the asymptotic size of the largest packing of near-maximal cliques remains an intriguing open question.  We return to this subject below, see Conjecture~\ref{conj:lower} and the discussion in Section~\ref{sec:upper-bound-discussion}.
	
	Let us briefly comment on what happens for cliques which are even closer to the maximum size, with $k=k_0-C$ for $C \in \{0,1,2\}$ and $p\in (0,1)$ constant.
	It is not hard to show that if $k= k_0-C$ and $E_p(n,k) < n^{2-\eps}$, for some $\eps >0$, then the size of the maximum  $k$-clique packing in $G(n,p)$ is of order $E_p(n,k)$ with probability at least $1-o_{p}(1)$.
	This happens because in this regime very few edges belong to more than one $k$-clique, and hence most $k$-cliques are edge-disjoint.

In fact, we prove a more general theorem (Theorem \ref{thm:main}), which includes Theorem \ref{thm:C} and extends it in three respects, which includes
 \begin{itemize}
 \item the case $C=3$,
 \item the case where $k_0-k$ goes to infinity, and
 \item the case $p=o(1)$, as long as $p=n^{o(1)}$.
 \end{itemize}

	Let us now discuss the challenges associated with the case $C=3$.  We shall see that the parameter $\gamma=\gamma(p,n,k)$ given by
	\[\gamma := \dfrac{\log E_p(n,k)}{\log n}\]
 plays a central role.
	The advantage of introducing this parameter is that we can write $E_p(n,k)=n^{\gamma}$.  
	The relation between $C$ and $\gamma$ is that we have 
 \eq{Cgamma}
 C-1+o(1)\, \le \, \gamma\,  \le\, C+o(1)
 \eqe
 see Lemma~\ref{lemma:expectation-k-c}, in the appendix, for a proof.
 	Theorem~\ref{thm:main} will include the case where $\gamma$ is bounded away from $2$, which includes the case $C=3$ almost completely.  
    We say ``almost completely'' because it is possible to choose an increasing subsequence $(n_i)_{i \ge 1}$ and set $p = n_i^{o(1)}$ so that $\gamma(p,n_i,k_0-3)$ tends to $2$ as $i$ tends to infinity.
 
    Now we observe that the parameter $\gamma=\gamma(p,n,k)$ is bounded above by a constant either when $k$ is small (and hence $\gamma$ is approximately $k$) or $k$ is near-maximum. 
    We briefly explain why.
    Let $k$ be such that $1\ll k \le k_0$ and $k_0-k \gg 1$. 
    Then, the expected number of $k$-cliques is
\begin{align*}
    \binom{n}{k}p^{\binom{k}{2}} \ge \left (\dfrac{p^{\frac{k-1}{2}}n}{k} \right )^k \ge \left (p^{\frac{k-k_0}{2}} \cdot \dfrac{p^{\frac{k_0-1}{2}}n}{k_0} \right )^k \ge \left (\Theta \left (p^{\frac{k-k_0}{2}} \right ) \right )^k.
\end{align*}
The claim follows by noticing that either $k_0-k$ or $k$ is at least $\frac{1}{4}\log_{1/p} n$.
As our results focus on the case where $k$ is near-maximum, we therefore impose the condition that $k \ge \log_{1/p} n$.

    We now formally define the cases covered by our more general theorem.  Given $\gamma_->2$ and $p_+\in (0,1)$, let us define $\Lambda_{\gamma_-,p_+}$ to be the family of all triples $(p,\gamma,k)$ where:
	\begin{enumerate}
		\item[(i)] $p=p_n$ is a sequence with $p_n \in (0,p_{+})$ for all $n\ge 1$ and $\log{p_n}/\log{n}\to 0$ (i.e., $p=n^{o(1)}$),
		\item[(ii)] $k=k_n$ is a sequence such that $k\ge \log_{1/p}{n}$,
		\item[(iii)] $\gamma=\gamma_n$ is the sequence $\gamma_n=\gamma(p_n,n,k_n)$, and $\gamma_n\ge \gamma_-$ for all $n\ge 1$. 
	\end{enumerate}
	Observe that if $(p,\gamma,k) \in \Lambda_{\gamma_-,p_+}$, then $\log_{1/p} n \le k \le k_0 \le 2.1 \log_{1/p} n$, for all sufficiently large $n$, as $E_p(n,k)\le 1$ if $k>k_0$.
 	
	\begin{theorem}\label{thm:main}  
		Let $\gamma_- >2$ and $p_+\in (0,1)$, and let $(p,n,k)\in \Lambda_{\gamma_-,p_+}$.  Then, with high probability, the random graph $G(n,p)$ contains at least
		\[
		\min \left\{\dfrac{\min\{\gamma-2,1\}pn^2\log{n}}{40k^4}\phantom{\bigg|}, \dfrac{pn^2}{2k^2}\right\}
		\]
		edge-disjoint $k$-cliques.
	\end{theorem}

	For $p \in (0,1)$ and $C\in \mathbb{N}_{\ge 4}$ constants and $k = k_0 -C$, we have $C-1+o(1) \le \gamma \le C+o(1)$ (see Lemma~\ref{lemma:expectation-k-c}).
    Therefore, Theorem~\ref{thm:main} immediately implies Theorem~\ref{thm:C}.
	To prove Theorem~\ref{thm:main} we consider a random graph process which sequentially selects and deletes $k$-cliques. 
	
 The minimum in Theorem~\ref{thm:main}  is necessary and states that either the $k$-clique removal process deletes half of the edges of $G(n,p)$, or the process lasts for at least {$\min\{\gamma-2,1\}pn^2\log{n}/(40k^4)$} steps.  In some sense this result partially vindicates the original conjecture of Alon and Spencer.
 More precisely, when $p \le e^{-\Omega(\sqrt{\log n})}$, it is possible to cover a positive proportion of the edges of $G(n,p)$ with edge-disjoint near-maximal cliques.  The threshold $e^{-\Theta(\sqrt{\log n})}$ is relevant as $pn^2/(2k^2)$ is the minimum if $p \le e^{-C\sqrt{\log n}}$, and {$\min\{\gamma-2,1\}pn^2\log{n}/(40k^4)$}
 is minimum if $p \ge e^{-c\sqrt{\log n}}$, for some constants $C,c>0$.

	We also provide upper bounds on the size of the largest $k$-clique packing for $k = k_0-C$, for every constant $C \in \mathbb{N}_{\ge 3}$.
	Although upper bounds were already obtained by Acan and Kahn~\cite{AK},
	we include a new proof which gives an upper bound with a linear dependence on $\gamma$ rather than an exponential dependence as in~\cite{AK}.
	For simplicity, we first consider only the case where $C \ge 4$.
	At the end of this section, we also state a general version of our theorem which hold under mild assumptions when $\gamma$ is not constant and $p$ goes to $0$.

	\begin{theorem}\label{thm:intro:upper:weak}
		Let $p\in (0,1)$ and $C \in \mathbb{N}_{\ge 4}$ be constants and $k=k_0-C$.   Then, except with probability at most $\exp(-n^{1+o(1)})$,
		the random graph $G(n,p)$ does not contain any family of 
		\[
		\dfrac{5(C - 2)}{1-p} \cdot \dfrac{pn^2 \log n }{k^4}
		\]
		edge-disjoint $k$-cliques.
	\end{theorem}

	The remainder of the introduction consists of subsections on the intuition behind the proof of Theorem~\ref{thm:main}; a discussion on upper bounds, including Theorem~\ref{thm:intro:upper:weak} and the more general version Theorem~\ref{thm:upper}; a quick overview of related articles; and the plan for the remainder of the article.

 \subsection{Intuition for the lower bounds}
	
	Here we discuss the heuristic behind the proof of Theorem~\ref{thm:main}, without too many details. 
	We will return to it in Section~\ref{sec:over}, where we give a detailed overview of the proof.  Let $G_0\sim G(n,p)$, and let
 \[
 m_1\, :=\, \dfrac{\min\{\gamma-2,1\}pn^2\log{n}}{40k^4} \quad \text{and} \quad m_{2}\,:=\, \dfrac{pn^2}{2k^2}\, ,
 \]
so that Theorem~\ref{thm:main} states that $G_0$ contains at least $\min\{m_1,m_2\}$ edge-disjoint $k$-cliques with high probability.

In order to simplify the presentation of this heuristic overview we restrict our attention to the case in which $m_1\ll m_2$.  In particular, in this case the number of edges removed in $m_1$ steps is $m_1\binom{k}{2}\ll pn^2/2\approx e(G_0)$.
	
Our approach involves a random process $(G_m)_{m= 0}^{M}$ which starts with a graph $G_0\sim G(n,p)$, and from which we remove a uniformly selected $k$-clique $K$ at each step. 
	The process halts when there are no $k$-cliques left to remove.
	This produces a sequence of graphs $(G_m)_{m= 0}^{M}$, where the last graph of the sequence,  $G_M$, is a $K_k$-free graph.
	Note that $M$, the number of steps of the process, is a random variable.  
	Moreover, the $M$ cliques removed during the process are all edge-disjoint, and hence our aim is to lower bound $M$.
	Thus, to prove Theorem~\ref{thm:main}, in the case $m_1\ll m_2$, it suffices to show that $M\ge m_1$ with high probability. 
	
	In order to show that the process lasts at least $m_1$ steps, we study the evolution of the number of $k$-cliques in the current graph $G_m$, which we denote by $Q(G_m)$.
	We first show that $Q(G_0)$ is concentrated around its mean $\Ex{Q(G_0)}=n^{\gamma}$.  
	We then analyse how the sequence $Q(G_m)$ evolves. 

 What is our heuristic for the evolution of $Q(G_m)$?
	At each step, a uniformly random $k$-clique $K$ is selected and removed.  However, this does not just affect this one clique.  Every clique which shares an edges with $K$ is also destroyed. 
	To `guess' how many cliques we destroy in each step, we consider the average number of cliques that contain a certain edge.
	
	In $G_m$, the average number of $k$-cliques on an edge is
	\[\dfrac{1}{e(G_m)}\sum_{e \in G_m} \sum_{K \se G_m}1_{ \{K \ni e\}}\, =\, \dfrac{\binom{k}{2}Q(G_m)}{e(G_m)}.\]
	The second sum above is over all $k$-cliques in $G_m$.
	The equality follows by interchanging the sums and noting that each $k$-clique is counted $\binom{k}{2}$ times.
	As the average edge is in $\binom{k}{2}Q(G_m)/e(G_m)$ such $k$-cliques, this allows us to `guess' that removing a $k$-clique from $G_m$ will lead to the destruction of approximately
	\[
	\binom{k}{2}\cdot \dfrac{\binom{k}{2}Q(G_m)}{e(G_m)}\, =\, \dfrac{\binom{k}{2}^2Q(G_m)}{e(G_0)-m\binom{k}{2}}
	\]
	$k$-cliques in total.
	Putting the pieces together, we `guess' that $Q(G_m)$ ought to evolve approximately like the sequence
	\begin{align*}
	\Ex{Q(G_0)}\prod_{i=0}^{m-1}\left(1\, -\, \frac{\binom{k}{2}^2}{e(G_0)-i\binom{k}{2}}\right) \, &\approx\, n^{\gamma}\exp \left (  -{\binom{k}{2}}^2 \,\,\, \sum\limits_{i=0}^{m-1} \dfrac{1}{e(G_0)-i \binom{k}{2}} \right)\\
 & \approx\, n^{\gamma}\exp \left(-\dfrac{k^4m}{2pn^2}\right) .\phantom{\Bigg|}
	\end{align*}
	We shall use so-called Differential Equation Method to prove that $Q(G_m)$ does indeed stay close to this trajectory for $m_1$ steps.  See Section~\ref{sec:over} for a complete overview of the details.

 For how long should the process last?  One might hope to prove an asymptotically best possible result using the same method.  Once the number of remaining cliques is much less than $n^2$, then the process ends quickly\footnote{in particular, when the number of cliques is less than $n^{2-\eps}$ then there are at most $n^{2-\eps}$ more steps}.  Therefore our best guess is that the process should continue until 
 \[
  n^{\gamma}\exp \left(-\dfrac{k^4m}{2pn^2}\right) \, =\, n^{2+o(1)}\, .
 \]
 This leads us to the following conjecture
.

\begin{conjecture}\label{conj:lower}
Let $p\in (0,1)$ and $C \in \mathbb{N}_{\ge 3}$ be constants and let $k=k_0-C$.  Then, with high probability, the random graph $G(n,p)$ contains at least 
		\[
		\frac{2(\gamma-2-o(1))pn^2\log{n}}{k^4}
		\]
		edge-disjoint $k$-cliques.
\end{conjecture}

We remark that even the most optimistic application of our upper bound result, Theorem~\ref{thm:upper}, produces an upper bound which is twice as large as this value.  We revisit this discussion after the statement of Theorem~\ref{thm:upper}.

	\subsection{A discussion about upper bounds}\label{sec:upper-bound-discussion}
	
	For positive integers $n$, $k$ and $t$, let $\NN(n,k,t)$ be the number of ways of selecting a set of $t$ edge-disjoint $k$-cliques in $K_n$.
	As in~\cite{AK},
	let $\zeta(n,k,t)$ be the probability that a sequence of $t$ cliques of size $k$ drawn uniformly and independently from $K_n$ are edge-disjoint.
	Note that $\NN$ and $\zeta$ are related by the equation
	\[\NN(n,k,t)\, =\, \dfrac{1}{t!}\binom{n}{k}^t \zeta(n,k,t).\]
	Acan and Kahn~\cite{AK} proved that there is a fixed $\alpha>0$ such that the following holds.
	If $1 \ll k \ll \sqrt{n}$ and $t = Dn^2/k^3$, then 
	\[\zeta(n,k,t) \le \exp\big(- \alpha (\log D)tk\big).\]
	We emphasize that in their result $D$ is allowed to depend on $n$ and $p$.
	Note that the expected number of sets of $t$ edge-disjoint $k$-cliques in $G(n,p)$ is $\NN p^{t\binom{k}{2}} = \zeta n^{\gamma t}/t!$, which tends to zero if $k \sim 2\log_{1/p}n$ and $D = \exp\big(2\gamma \alpha^{-1}\log\big(\frac{1}{p}\big)\big)$.
	This shows that if $p\in (0,1)$ and $C>2$ are constants and $k=k_0-C$, then with high probability one can find at most $2^{O_p(\gamma)}\cdot n^2/k^3$ edge-disjoint $k$-cliques in $G(n,p)$.

	One way to improve the upper bound on the size of the maximum $k$-clique packing in $G(n,p)$, is to obtain better upper bounds on $\zeta$.
	Let $(K^1,K^2,\ldots,K^t)$ be a sequence of $k$-cliques drawn uniformly and independently from $K_n$.
	If $1 \ll k \ll \sqrt{n}$, then the probability that $|V(K^i)\cap V(K^j)|\ge 2$ is asymptotically $k^4/(2n^2)$, for $i\neq j$.
	As Acan and Kahn suggested in~\cite{AK},
	if the events $\{|V(K^i)\cap V(K^j)|\ge 2\}$, for $1\le i<j\le t$, were close to be mutually independent, then we would have
	\[
	\zeta(n,k,t)\, \le \, \left(1-\dfrac{k^4}{2n^2}\right)^{\binom{t}{2}}\, = \, \exp\left(\frac{-(1-o(1))t^2k^4}{4n^2}\right)\, .
	\]
	In fact, Acan and Kahn showed that this holds in the case where $1\ll k\ll n^{1/2}$ and $1\ll t\ll n^2/k^3$.

	We shall need to consider slightly larger values of $t$.  
	Define
	\[t_0:= \dfrac{5(\gamma-2)}{1-p} \cdot \dfrac{pn^2 \log{n}}{k^4}\]
	and define $\beta=\beta(n,k)$ to be maximal such that the inequality
	\eq{beta}
	\zeta(n,k,t)\, \le \, \exp\left(\frac{-\beta t^2k^4}{n^2}\right) 
	\eqe
	holds for all $t\le t_0$.  The result of Acan and Kahn mentioned above shows that~\eqr{beta} holds with $\beta=1/4+o(1)$ for $t\ll n^2/k^3$, and with $\beta>0$ for $t$ of order $n^2/k^3$, which is of the same order as $t_0$ in cases where $k=\Theta(\log{n})$.
	This includes all cases corresponding to near-maximal cliques in dense random graphs (with $p$ constant).
	
	Our upper bounds on the size of the maximum $k$-clique packing in $G(n,p)$ depend on $\beta$.
	Nevertheless, even in the worse case scenario in which $\beta$ is asymptotically zero, our result is already non-trivial.
	In particular, our result gives an independent proof of Acan and Kahn's result, with a linear dependence in $\gamma$ rather than exponential.
	Our results also hold under mild assumptions when $\gamma$ is not constant and $p$ goes to $0$.

	\begin{restatable}{theorem}{upper}\label{thm:upper}
		Let $\gamma_- >2$ and $\eps, p_+\in (0,1)$, and let $(p,n,k)\in \Lambda_{\gamma_-,p_+}$.
		Then, except with probability at most $\exp(-n^{1+o(1)})$, 
		the random graph $G(n,p)$ does not contain any family of 
		\[
		\dfrac{(\gamma -2)(4+\eps)}{1+(4\beta-1)p} \cdot \dfrac{pn^2}{k^4}\log n
		\]
		edge-disjoint $k$-cliques.
	\end{restatable}
	
	Our proof relies on showing that the expected number of such collections of edge-disjoint $k$-cliques in $G(n,m)$ tends to zero, where $m = \lceil p\binom{n}{2}+(pn^2)^{3/4}\rceil$.  The theorem then follows by relating the $G(n,m)$ and $G(n,p)$ models in the standard way.

	Observe that Theorem~\ref{thm:upper} immediately implies Theorem~\ref{thm:intro:upper:weak}, as $\eps < 1$ and we have $\beta \ge 0$ and $C\ge \gamma+o(1)$, as we saw above in~\eqr{Cgamma}.
    As with Theorem~\ref{thm:main}, Theorem~\ref{thm:upper} does not include the case where $\gamma$ tends to 2.
	
	What is the best we could hope to achieve using this upper bound?  If~\eqr{beta} holds with $\beta=1/4+o(1)$, then Theorem~\ref{thm:upper} implies that the largest family of edge-disjoint $k$-cliques has size at most $(4+o(1))(\gamma-2)pn^2\log{n}/k^4$ with high probability.  Recall that, even with our most optimistic conjecture (Conjecture~\ref{conj:lower}), our approach to the lower bound is only likely to produce a family of $(2+o(1))(\gamma-2)pn^2\log{n}/k^4$ edge-disjoint $k$-cliques.  
	These possible bounds differ by a factor of $2$, and we do not know whether either of these bounds (if true) would be best possible.  It is quite possible that our analysis of the clique removal process breaks down before the process ends, and it might even be possible that the process lasts $(4+o(1))(\gamma-2)pn^2\log{n}/k^4$ steps, but analysis of the process would seem to be extremely difficult once most edges are no longer in any $K_k$.

	\subsection{Historical background on the Differential Equation Method}
	
	Our approach relies on tracking various random variables related to the evolution of a random graph process.  This method is known as the Differential Equation Method, as the trajectories are generally given by solutions to differential equations.  Although, as we `guess' the trajectories, we do not actually need to work with differential equations at all.
	
	The Differential Equation Method was popularised in the combinatorial community by Wormald~\cite{Worm95,Worm99}.  In particular, Wormald~\cite{Worm99} gives a general theorem which may be used in many applications of the method.  The roots of the method may be traced back further, as can be seen in the references given in~\cite{Worm99}.  It is especially worthwhile highlighting Kurtz~\cite{K} and Karp and Sipser~\cite{KS}.
	
	Since then the method has seen many applications.  It has been used to understand the $H$-free and triangle-free random graph processes, see~\cite{Boh,BK1,BK2,FGM} (and the references therein).
	The line of research most directly related to our results is the study of the triangle removal process.  
	The triangle removal process is similar to the process we study except with $k=3$, as a triangle is removed each time.  
	That said, there are significant differences in the analysis of the processes due to the fact that our cliques have unbounded size $k=\Theta(\log{n})$ rather than $k=3$. 
	Moreover, while our main concern is how long the process lasts asymptotically, the main question considered about the triangle-free process is the question of how many edges remain when the process ends (becomes triangle-free).  Bollobás and Erd\H os conjectured in 1990 that order $n^{3/2}$ edges remain with high probability.  Progress began with the bound $o(n^2)$ due to Spencer~\cite{Str} and independently by Rödl and Thoma~\cite{RTtr}.  Grable~\cite{Gtr} then improved this upper bound to $n^{11/6+o(1)}$ and outlined how the argument should give $n^{7/4+o(1)}$.  Then, in 2015, Bohman, Frieze and Lubetzky~\cite{BFL} made a major breakthrough, improving the upper bound to $n^{3/2+o(1)}$.
	
	Let us also mention two survey type articles on the  Differential Equation Method, due to Díaz and Mitsche~\cite{DM} and Bennett and Dudek~\cite{BD}.  Finally, we mention a recent shorter proof of Wormald's Theorem by Warnke~\cite{warnkeL}.
	
	\subsection{A brief history of packing problems}
	
	The history of clique packings dates back to the work of Pl\"ucker~\cite{Plucker} in 1853, who showed that there exists a collection of $\big(\frac{1}{3}-o(1)\big)\binom{n}{2}$ edge-disjoint triangles in $K_n$.
	This is asymptotically best possible, as each triangle ocuppies three edges.
	Wilson~\cite{Wilson} extended this result in 1975, more than one hundred years later, to cliques of arbritrary size.
	Actually, these results were much more precise.
	They showed that if $n$ is sufficiently large and certain divisibility conditions are satisfied, then $K_n$ contains a \emph{perfect} $k$-clique packing.
	That is, there exists an $k$-clique packing whose union of edges is equal to $\binom{n}{2}$.
	In 1985, Frankl and R\"odl~\cite{FR85} extended these results asymptotically by showing that for a fixed graph $H$ and $n$ sufficiently large, there exists a \emph{nearly} perfect $H$-packing in $K_n$.
	In other words, we can find $(1-o(1))\binom{n}{2}/e(H)$ edge-disjoint copies of $H$ in $K_n$.
	Much later, in a major breakthrough on the existence of designs, Keevash \cite{Keevash,Keevash2} and, independently, Glock, Kühn, Lo, and Osthus \cite{GKLO23} determined the necessary and sufficient conditions for a perfect $H$-packing in $K_n$. They showed that, for all sufficiently large $n$, the obvious divisibility constraints on $n$ are also sufficient to guarantee an $H$-packing of size exactly $\binom{n}{2}/e(H)$. A corresponding result was proved for hypergraphs as well.

	Frankl and R\"odl~\cite{FR85} pioneered the study of packings in random graphs. They showed that for $p \ge n^{\eps-2/(k+1)}$, with $\eps>0$, we have a nearly perfect $K_k$-packing in $G(n,p)$ with high probability.
	More recently, there has been a focus on results of packing  with large (sometimes spanning) subgraphs. 
	Moreover, many of these results generalise to pseudo-random graphs.
	In order to aid readability, we shall state all results as with high probability and in $G(n,p)$.
	
	In 2005, Frieze and Krivelevich~\cite{FK05} showed that if $p \gg n^{-1}\log n$, then $G(n,p)$ contains a near perfect packing of hamiltonian cycles, with high probability.
	In 2010, the same authors extended this result to hypergraphs~\cite{FK12}.
	In 2016, the R{\"o}dl nibble method was brought by Böttcher, Hladký, Piguet and
	Taraz~\cite{BHPT16} to solve the approximate version of the tree packing conjecture in complete graphs.
	This result was further extended through a series of papers~\cite{FLM17,JKKO19,MRS16}.
	These works finally culminated in the work of  Kim, K{\"u}hn, Osthus and Tyomkyn~\cite{KKOT19} who showed that if $p$ is constant, then with high probability there exists a near perfect $H_n$-packing in $G(n,p)$ for every bounded degree graph $H_n$ on $n$ vertices.
	
	The first result to consider packings of graphs with unbounded degrees is due to Ferber and Samotij~\cite{FS19}.
	They showed that if $p \ge n^{-1}(\log n)^{36}$ and $T_n$ is a tree on $n$ vertices with degree bounded by $(np)^{1/6}$, then with high probability $G(n,p)$ contains a near perfect $T_n$-packing.
	Allen, B{\"o}ttcher, Hladk{\`y} and Piguet~\cite{ABHP19} showed that if $H$ is a $D$-degenerate graph with degree bounded by $O(n/\log n)$, then there exists a near perfect $H$-packing in $K_n$.
	By building on~\cite{ABHP19}, Allen, Böttcher, Clemens and Taraz~\cite{ABCT22} showed that if $p$ is constant, then with high probability $G(n,p)$ contains a near perfect $H_n$-packing provided that $H_n$ has bounded degenerancy, is not close to being spanning and has maximum degree $o(n/\log n)$.
	In 2021, Allen, 
	B{\"o}ttcher, Clemens, Hladk{\`y}, Piguet
	and Taraz~\cite{ABCHPT21} extended the last result by removing the bounded degenerancy assumption in the case of trees.

	In 2020, Keevash and Staden~\cite{KS20} and, independently,  Montgomery, Pokrovskiy and Sudakov~\cite{MPS21}, proved the longstanding Ringel’s tree packing conjecture.
	The method of Keevash and Staden was also applied for packing trees in $G(n,p)$ with no condition on the maximum degree of the tree, but only on the number of edges.
	More recently, Decourt, Kelly and Postle~\cite{DKP24} used the absorption method to show, among other results, that if $q \in \mathbb{N}_{\ge 4}$ and $p \ge n^{-\frac{1}{q+0.5}+\beta}$ for some $\beta >0$, then with high probability $G(n,p)$ contains packing of $q$-cliques containing all but $(q-2)n+O(1)$ edges. 
	This refines the result of Frankl and R\"odl in~\cite{FR85}, where the bound on the number of edges which are not in a maximum packing is far from linear.

	\subsection*{Plan of the article}
	
	In Section~\ref{sec:over} we give an overview of the proof, including discussion of the random graph process and the variables we shall track, $Q(G_m)$ and $Y_e(G_m)$.  We also introduce some useful inequalities.  In Section~\ref{sec:I} we consider the behaviour of $Q(G_0)$, $Y_e(G_0)$ and other quantities in the initial random graph $G_0\sim G(n,p)$.  We then prove the main concentration results for the evolution of the sequences $Q(G_m)$ and $Y_e(G_m)$ in Sections~\ref{sec:Q} and~\ref{sec:Y}, respectively.  Finally, in Section~\ref{sec:upper} we prove the upper bound result, Theorem~\ref{thm:upper}.

	\section{Overview of the proof}\label{sec:over}
	As we mentioned in the introduction, we prove Theorem~\ref{thm:main} by considering the following random graph process.  We start with an initial graph $G_0$, where $G_0\sim G(n,p)$, and then, at each step of the process, we remove a uniformly random $k$-clique from the current graph.  That is, we obtain $G_{m+1}$ from $G_m$ by selecting a uniformly random $k$-clique $K\subseteq G_m$ and setting $E(G_{m+1}):=E(G_m)\setminus E(K)$.
	
	We may immediately observe that the number of edges in $G_m$ is precisely $e(G_m)=e(G_0)-m\binom{k}{2}$, as precisely $\binom{k}{2}$ edges are removed in each step.
	It will also be possible (with more work!) to show that other parameters of the graph $G_m$ are also well behaved, in the sense that they remain close to certain pre-defined trajectories.
	For example, $Q(G_m)$, which is the number of $k$-cliques remaining in the graph $G_m$, will remain close to the trajectory
	\begin{align}\label{def:tq}
		\tq(m)\, :=\, \widetilde{Q}(0)\prod_{i=0}^{m-1}\left(1\, -\, \frac{\binom{k}{2}^2}{e(G_0)-i\binom{k}{2}}\right) \, \approx\, \widetilde{Q}(0)\exp \left (  -{\binom{k}{2}}^2 \sum\limits_{i=0}^{m-1} \dfrac{1}{e(G_0)-i \binom{k}{2}} \right) 
	\end{align}
	where $\widetilde{Q}(0)\coloneqq \mathbb{E}[Q(G_0)]=E(n,k)=\binom{n}{k}p^{\binom{k}{2}}=n^{\gamma}$.    
	
	Let us be more specific about how close $Q(G_m)$ should remain to $\tq(m)$.  Define
	\begin{align}\label{eq:delta}
		\delta:=\dfrac{\min\{\gamma-2,1\}}{10}.
	\end{align}
	The allowed relative error for $Q(G_m)$, relative to $\tq(m)$, will be 
	\begin{align}\label{eq:gQerror}
		g_Q(m)\,: =\, 2n^{-\delta}\prod_{i=0}^{m-1}\left(1\, +\, \frac{\binom{k}{2}^2}{e(G_0)-i\binom{k}{2}}\right)\, \approx\, 2n^{-\delta} \exp \left (  \binom{k}{2}^2 \sum\limits_{i=0}^{m-1} \dfrac{1}{e(G_0)-i \binom{k}{2}} \right ),
	\end{align}
	where $g_Q(0)=2n^{-\delta}$.
	This choice of error function is related to the techniques we use in the analysis of the random process, which are based on supermartingales and related inequalities.
	At the end of Section~\ref{sub:the-method} we briefly explain the motivation behind our choice.
	
	We are now nearly ready to state our result about the evolution of $Q(G_m)$.  We use the notation $x = a\pm b$ for $x\in [a-b,a+b]$.  Let us set 
	\begin{align}\label{eq:defn-m-star}
		m_{*}:=\min \left\{\left\lfloor  \dfrac{\delta pn^2\log{n}}{4k^4}\right\rfloor, \dfrac{pn^2}{2k^2}\right\} .
	\end{align}  
	
	\begin{theorem}\label{thm:Q}  Let $\gamma_- >2$ and $p_+\in (0,1)$, and let $(p,n,k)\in \Lambda_{\gamma_-,p_+}$. 
		Then, with high probability, in the $k$-clique removal process $(G_m)_{m\ge 0}$ with random initial graph $G_0\sim G(n,p)$ we have
		\[
		Q(G_m)\, = \, \tq(m)\, \left(1\pm g_Q(m)\right)
		\]
		for all $m < m_*$.
	\end{theorem}
	
	Our main theorem, Theorem~\ref{thm:main}, follows immediately from Theorem~\ref{thm:Q}.  
	
	\begin{proof}[Proof of Theorem~\ref{thm:main}, assuming Theorem~\ref{thm:Q}]
		The $k$-clique removal process $(G_m)_{m\ge 0}$ continues to run as long as $Q(G_m)>0$.
		By Theorem~\ref{thm:Q}, it suffices to show that $ g_{Q}(m) < 1$ and $\tq(m)>0$ for all $m < m_{*}$.
		As $g_Q$ is increasing and $\tq$ is decreasing, it suffices to show this for $m=m_{*}$.
		As $e(G_0)=(1+o(1))pn^2/2$ with high probability, by Chernoff's inequality\footnote{Chernoff's inequality states that if $Y$ is a binomial random variable and $t\ge0$, then $\Pr(|Y-\Ex{Y}|\ge t)\le 2e^{-t^2/(2\Ex{Y}+t)}$.}, and $m_{*} \le pn^2/(2k^2)$, we have $e(G_{m_{*}})\ge (1+o(1))pn^2/4$, and hence $g_Q(m_{*})$ is at most
		\[
		g_Q(m_{*})\le 2n^{-\delta} \exp \left (  \frac{(1+o(1))k^4 m_{*}}{pn^2}\right).
		\]
		As $m_{*} \le \delta pn^2\log n/(4k^4)$, it follows that $g_Q(m_{*})\, \le\, 2n^{-3\delta/4}\, \ll\, 1$
		with high probability.  As the bracketed term in the definition of $\tq(m_ *)$ is positive (as $e(G_{m_{*}})\ge  (1+o(1))pn^2/4>k^4$), we also have $\tq(m_*)>0$.
	\end{proof}
	
	Now that we have shown that Theorem~\ref{thm:Q} implies the main theorem, we will discuss how to analyse the evolution of the process and prove Theorem~\ref{thm:Q}.  
	One part of that challenge is to understand the nature of the one-step changes in $Q(G_m)$.  
	The one-step change in $Q(G_m)$ corresponds to the number of $k$-cliques which are ``destroyed'' at each step. Note that one $k$-clique $K$ is removed as part of this process, but naturally this ``destroys'' all other $k$-cliques $K'$ which share an edge with $K$.  This leads us to consider for each edge $e\in E(G_m)$ the number of $k$-cliques which contain this edge.  
	
	For every edge $e\in E(K_n)$ we may define $Y_e(G_m)$ to be the number of $k$-cliques in $G_m \cup \{e\}$ which contain $e$.
	For the reasons discussed above, it will also be essential to track the sequences $Y_e(G_m)$, simultaneously for all $e\in E(K_n)$.  Our aim will be to show that $Y_e(G_m)$ remains close to the function
	\begin{align}\label{def:ty}
		\ty(m):=\binom{k}{2} \dfrac{\tq(m)}{e(G_m)}.
	\end{align}
	We may also extend this definition to sets of three vertices $S\subseteq V(G_0)$.  Given $S$ a set of three vertices, let $Y_S(G_m)$ be the number of $k$-cliques in $G_m \cup \binom{S}{2}$ containing $S$, where $\binom{S}{2}:=\{A\se S: |A|=2\}$.
	In fact, we shall not require a close control of the evolution of $Y_S(G_m)$. It will be sufficient for us to bound $Y_{S}(G_0)$ and use monotonicity $Y_S(G_m)\le Y_S(G_0)$ for all $m\ge 0$.

	Let us begin with the properties we require of the initial graph $G_0\sim G(n,p)$.  


	\begin{prop}\label{prop:refined}\label{prop:I} Let $\gamma_- >2$ and $p_+\in (0,1)$, and let $(p,n,k)\in \Lambda_{\gamma_-,p_+}$.
		With high probability, all the following items are satisfied: 
		\begin{enumerate}
			\item [$(1)$] $e(G_0)\, = \, pn^2 \pm n^{3/2}$;
			\vspace*{0.2cm}
			\item [$(2)$] $Q(G_0)\, = \, (1\pm n^{-\delta})\tq(0)$;
			\vspace*{0.2cm}
			\item [$(3)$] $Y_e(G_0)\, = \, (1\pm n^{-\delta})\ty(0)$ for all $e\in E(K_n)$;
			\vspace*{0.2cm}
			\item [$(4)$]\label{prop:YS} $Y_S(G_0)\, \le\, n^{\delta}\max\{1,\Ex{Y_{S}(G_0)}\}$ for all sets $S \se V(G_0)$ with $|S| =3$.
		\end{enumerate}
	\end{prop}
	%
	%
	%
	We prove Proposition~\ref{prop:I} in Section~\ref{sec:I}.  Now, what remains is to prove that the evolution of $Q(G_m)$ follows the trajectory claimed in Theorem~\ref{thm:Q}.   Let $\tau_Q$ be the minimum of $m_{*}$ and the first value of $m$ such that $Q(G_m)$ leaves its allowed interval $\tq(m)(1\pm g_Q(m))$.  
	It clearly suffices to show that $\pr{0<\tau_Q < m_*}\to 0$ as $n\to \infty$, as this, together with Proposition~\ref{prop:I}, implies that $Q(G_m)$ remains in the required interval for all $m < m_*$.
	
	We also consider a stopping time $\tau_Y$ related to the behaviour of the $Y_e(G_m)$.  We set
	\begin{align}\label{eq:gY}
		g_Y(m)\, :=\, 10n^{-\delta}\prod_{i=0}^{m-1}\left(1\, +\, \frac{2\binom{k}{2}^2}{e(G_0)-i\binom{k}{2}}\right)\, \approx\, 10n^{-\delta} \exp \left (  2\binom{k}{2}^2 \sum\limits_{i=0}^{m-1} \dfrac{1}{e(G_0)-i \binom{k}{2}} \right )
	\end{align}
	to be the allowed relative error of the random variables $Y_e(G_m)$.  It is not hard to check that
	\begin{align}\label{eq:boundongQgY}
		n^{-\delta }\,\le\, g_Q(m)\, \le\, g_Y(m)\, \le\, n^{-\delta/4}
	\end{align}
	for all $m \le m_{*}$.
    These estimates are like in the proof of Theorem~\ref{thm:main}.
	For each edge $e\in K_n$, let $\tau_{Y_e}$ be the minimum of $m_{*}$ and the first value of $m$ such that $Y_e(G_m)\neq \ty(m)(1\pm g_Y(m))$.
	We now set
	\begin{align}\label{eq:stopping-times}
		\tau_Y=\min\{\tau_{Y_e}:e\in K_n\}\, ,
	\end{align}
	and set $\tau = 0$ if $G_0$ does not satisfy items (1)--(4) in Proposition~\ref{prop:I}, and otherwise
	\begin{align}\label{eq:defntau}
		\tau=\min\{\tau_Q,\tau_Y\}.
	\end{align} 
	One may think of $\tau$ as the stopping time which flags the first moment that ``something goes wrong''.  We shall prove the following two propositions.\vspace{2mm}
	
	\begin{prop}\label{prop:Q} $\pr{0<\tau=\tau_Q<m_{*}}\, \to\, 0$ as $n\to \infty$\end{prop}\vspace{1mm}
	
	\begin{prop}\label{prop:Y} $\pr{0<\tau=\tau_Y<m_{*}}\, \to\, 0$ as $n\to \infty$\end{prop}\vspace{1mm}
	
	These propositions are proved in Sections~\ref{sec:Q} and~\ref{sec:Y} respectively.  It may be observed that Theorem~\ref{thm:Q} follows from the Propositions~\ref{prop:I},~\ref{prop:Q} and~\ref{prop:Y}.  
 
 \begin{proof}[Proof of Theorem~\ref{thm:Q}]
 By Propositions~\ref{prop:I},~\ref{prop:Q} and~\ref{prop:Y}, each of the events $\tau=0$, $0<\tau=\tau_Q<m_{*}$ and $0<\tau=\tau_Y<m_{*}$ has probability $o(1)$.  And so, with high probability $\tau_Q=\tau=m_{*}$.  By considering the definition of $\tau_Q$ as the minimum of $m_*$ and the first time $Q(G_m)$ leaves its allowed interval $\tq(m)(1\pm g_Q(m))$ the theorem clearly follows.\end{proof}

	\subsection{The method of proof and concentration inequalities}\label{sub:the-method}

	Now that we have given an overview of the structure of the proof, let us say something about methods.
	
	In Section~\ref{sec:I}, we must prove inequalities related to the random variables $Q(G_0),Y_e(G_0)$ and $Y_S(G_0)$ in the random graph $G_0\sim G(n,p)$.  We begin by proving rough bounds up to a factor of $n^{\delta}$ using the moment method (effectively Markov's inequality applied to a large power of the random variable).  Equipped with these rough bounds we then prove more refined bounds using the deletion method~\cite{JR1,JR2}. See, for example, Theorem 2.1 in~\cite{JR2}.

	\begin{theorem}[The deletion method]\label{thm:deletion-method}
		Suppose that $\{Y_{\alpha}: \alpha \in \mathcal{A}\}$, is a finite family of non-negative random variables and that $\sim$ is a symmetric relation on the index set $\A$ such that each $Y_{\alpha}$ is independent of $\{Y_{\beta}:\beta \not \sim \alpha\}$; in other words, the pairs $(\alpha,\beta)$ with $\alpha \sim \beta$ define the edge set of a (weak) dependency graph for the variables $Y_{\alpha}$.
		Let $X := \sum_{\alpha} Y_{\alpha}$ and $\mu = \Ex{X}$.
		For $\alpha \in \A$, let $\tilde{X}_{\alpha}=\sum_{\beta \sim \alpha} Y_{\beta}$ and $X^{*} = \max_{\alpha \in \A} \tilde{X}_{\alpha}$.
		If $t>0$, then for every real $r>0$,
		\[\Pr(X \ge \mu+t)\le \left(1+\dfrac{t}{2\mu}\right)^{-r}+ \Pr \left(X^{*}> \dfrac{t}{2r}\right).\]
	\end{theorem}
	
	Another inequality that we shall use to prove refined bounds on $Q(G_0)$ and $Y_e(G_0)$ is Janson's inequality~\cite{Jan90,JLR}. See, for example, Theorem 2.14 in~\cite{JLR}.

	\begin{theorem}[Janson's inequality]\label{thm:janson}
		Let $p \in (0,1)$, $\Gamma$ be a finite set and $\Gamma_p$ be the binomial random subset of $\Gamma$. 
		Let $I_A$ be the random variable which is 1 if $A \se \Gamma_p$ and 0 otherwise.
		Let $\S$ be a family of non-empty subsets of $\Gamma$ and set $X = \sum_{A \in \S} I_A$. That is, $X$ is the number of sets $A \in S$ that are contained in $\Gamma_p$.
		Let $\bar{\Delta} = \sum_{(A,B):A\cap B \neq \emptyset} \Ex{I_AI_B}$, where the sum is over $(A,B)\in \mathcal{S}^2$ such that $A\cap B \neq \emptyset$.
		Then, for $0 \le t \le \Ex{X}$, we have
		\[\pr{X \le \Ex{X}-t} \le \exp \left(- \dfrac{t^2}{2 \bar{\Delta}}\right).\]
	\end{theorem}

	The proofs in Sections~\ref{sec:Q} and~\ref{sec:Y} are based on supermartingales and related inequalities.  In particular we shall use the well known Hoeffding--Azuma inequality~\cite{Azuma,Hoeff}.
	Before we state it, we need some definitions.
	Let $(\F_i)_{i=0}^{m}$ be a filtration, that is, a sequence of $\sigma$-algebras such that $\F_0 \se \F_1 \se \cdots \se \F_m$.
	We say that the discrete random process $(X_i,\F_i)_{i=0}^{m}$ is a supermartingale if $\Ex{X_i|\F_{i-1}}\le X_{i-1}$ for all $i \in \{1,\ldots,m\}$.
	The  Hoeffding--Azuma inequality can be stated as follows.
	
	\begin{lemma}[Hoeffding--Azuma inequality]\label{lem:hoeff-azuma}
		Let $(X_i,\F_i)_{i=0}^{m}$ be a supermartingale with increments satisfying $|X_{i}-X_{i-1}|\le c_i$ almost surely for all $i \in \{1,\ldots,m\}$.
		Then, for every $\alpha\ge 0$, we have
		\[
		\pr{X_m-X_0\, \ge \, \alpha}\, \le\, \exp\left(\frac{-\alpha^2}{2\sum_{i=1}^{m}c_i^2}\right)\, .
		\]
	\end{lemma}
	
	In some applications, when $|X_i-X_{i-1}|$ is typically much smaller than $\|X_i-X_{i-1}\|_{\infty}$, the following inequality of Freedman~\cite{F} gives a better bound.
	
	\begin{lemma}[Freedman's inequality]\label{lem:Freedman}
		Let $(X_i,\F_i)_{i=0}^{m}$ be a supermartingale and let $R \in \R$ be such that $\max_i |X_i-X_{i-1}| \le R$.
		Let $V(m)$ be the quadratic variation of the process until step $m$, that is,
		\[
		V(m)\, :=\, \sum_{i=1}^{m}  \Ex{ (X_i-X_{i-1})^2 | \F_{i-1} }.
		\]
		Then, for every $\alpha, \beta > 0$, we have
		\[
		\pr{X_m- X_0 \ge \alpha \mbox{     and     } V(m) \le \beta }\, \le\, \exp \left( \frac{-\alpha^2}{2(\beta + R\alpha)} \right).
		\]
	\end{lemma}
	
	There are then two main challenges involved in applying this inequality to the random variables we wish to track
	\begin{enumerate}
		\item[(i)] We must encode failure events as deviation events for supermartingales
		\item[(ii)] We must control the quadratic variation and give an almost sure bound for the magnitude of the increments of these supermartingales.
	\end{enumerate}
	
	These details are covered in Sections~\ref{sec:Q} and~\ref{sec:Y}.  We remark that (i) is related to the (conditional) expected change of the random variable in question, and it is for this reason that some control of the $Y_e(G_m)$ is necessary to control $Q(G_m)$.  
	With respect to (ii) the maximum increments are again related to the one-step change in our random variables.  This is why some control of $Y_S(G_m)$ for sets of size $3$ is necessary to understand the evolution of the $Y_e(G_m)$.  We do not do anything special to control the quadratic variation, we simply use information about the expected change and the maximum change to obtain a bound.
	
	At last, let us briefly explain how to to encode failure events as deviation events for supermartigales and how this relates to the choice of the error functions.
    For simplicity, in this discussion we only consider deviations of $Q(G_m)$ above its expected trajectory.
	Although the intuition we give here differs slightly from our analysis, it is instructive to keep this rough idea in mind.
	Let $\tau$ be the stopping time given in~\eqref{eq:defntau}.
	Define the random process $(X_m)_{m\ge0}$ as
	\[
	X_m\, :=\, \begin{cases} Q(G_{m})\, -\, (1+g_Q(m))\tq(m),& \text{ if } m\le \tau\\  
		Q(G_{\tau})\, -\, (1+g_Q(\tau))\tq(\tau),& \text{ if } m> \tau\,  .\end{cases}
	\]
	In order to use concentration inequalities such as the Hoeffding--Azuma and Freedman inequalities, we must show that $(X_m)_{m\ge 0}$ is a supermartingale.
	Suppose that $Q(G_m) = (1\pm g_Q(m))\tq(m)$ with very high probability.
	As the average number of $k$-cliques containing an edge in $G_m$ is equal to $\binom{k}{2}Q(G_m)/e(G_m)$, we expect the conditional expectation $\Ex{Q(G_{m+1})-Q(G_m)|G_m}$ to be roughly
	\begin{align}\label{eq:cond-Q}
		\Ex{Q(G_{m+1})-Q(G_m)|G_m} \approx 
		- \binom{k}{2}^2 \dfrac{Q(G_m)}{e(G_m)} 
		= - \binom{k}{2}^2 \dfrac{\tq(m)}{e(G_m)} (1\pm g_Q(m)).
	\end{align}
	On the other hand, we have
	\begin{align}\label{eq:cond-tq}
		\tq(m+1)-\tq(m) = - \binom{k}{2}^2 \dfrac{\tq(m)}{e(G_m)} 
	\end{align}
    By subtracting~\eqref{eq:cond-tq} from~\eqref{eq:cond-Q} and using the definition of $(X_m)_{m\ge 0}$, we expect to have
	\[\Ex{X_{m+1}-X_m|G_m} \approx \left (1 \, \pm \,  \dfrac{\binom{k}{2}^2}{e(G_m)} \right)(g_Q\tq)(m)- (g_Q\tq)(m+1).\]
	For $(X_m)_{m\ge0}$ to be a supermartingale almost surely, we must have $\Ex{X_{m+1}-X_m|G_m} \le 0$, which is equivalent to
	\[g_Q(m+1)\ge \left (1 \, + \,  \dfrac{\binom{k}{2}^2}{e(G_m)} \right)g_Q(m).\]
	Our error function $g_Q$ was chosen precisely so that the above inequality could be satisfied.
	\section{Initial values of the random variables}\label{sec:I}

	In this section we prove Proposition~\ref{prop:refined} about the behaviour of $e(G_0)$, $Q(G_0),Y_e(G_0)$ and $Y_S(G_0)$.
	By Chernoff's inequality, we have $e(G_0) \in pn^2 \pm n^{3/2}$ with high probability.
	Thus, we only need to focus on showing that the remaining variables are highly concentrated around their means.
	
	\subsection{Rough upper bounds using moments}\label{sec:Irough}
	
	To prove Proposition~\ref{prop:refined}, it is necessary to show a rough upper bound on $Y_S(G_0)$ for sets $S$ of size $3$ and show that $Q(G_0)$ and $Y_e(G_0)$ concentrate around their means.
	It will actually be helpful to first prove rough upper bounds for all these quantities. 
	
	\begin{prop}\label{prop:rough} Given $\eps, p_{+}\in (0,1)$ constants, $\gamma_{-} >2$ and $(p,n,k)\in \Lambda_{\gamma_-,p_+}$, for all sufficiently large $n$, there is probability at least $1-n^{-4}$ that
		\[
		Q(G_0)\, \le\, n^{\eps}\Ex{Q(G_0)}\quad ,\quad \max_{e\in E(K_n)}Y_{e}(G_0)\, \le\,n^{\eps}\Ex{Y_{e}(G_0)}
		\]
		and 
		\[
		\max_{S:|S|=3}Y_S(G_0)\, \le\, n^{\eps}\max\{1,\Ex{Y_{S}(G_0)}\}\, .
		\]
	\end{prop}
	
	Observe that item (1) in Proposition~\ref{prop:refined} follows from Chernoff's inequality and item (4) follows from 
	Proposition~\ref{prop:rough}. Items (2) and (3) are proven in Section~\ref{sec:Iprecise}.
	
	Proposition~\ref{prop:rough} follows easily from our next lemma (Lemma~\ref{lem:moments}), where we bound the moments
	of our main random variables. 
	Rather than prove the moment bounds separately for each of $Q(G_0)$, $Y_{e}(G_0)$ and $Y_{S}(G_0)$, we shall extend the use of the notation $Y_S$.
	For a graph $G$ and a set $S\subseteq V(G)$, we define $Y_S(G)$ to be the number of $k$-cliques contained in $G\cup \binom{S}{2}$ and containing $S$.
	Note that this includes $Q(G_0)$ and $Y_{e}(G_0)$, since these correspond to sets $S$ of size $0$ and $2$, respectively.
	Although we only need to bound the moments of $Y_{S}(G_0)$ for sets $S$ of size at at most $3$, we can show bounds on the moments of $Y_S(G_0)$ for every bounded set $S$.
	
	\begin{lemma}\label{lem:moments}
		Let $p_{+}\in (0,1)$ and $\ell,s\in \mathbb{N}_{\ge 0}$ be constants and $\gamma_->2$. Then, there exists $C>0$ such that the following holds. If $S \se [n]$ is a set of size $s$ and $(p,n,k)\in \Lambda_{\gamma_-,p_+}$, then
		\[
		\Ex{Y_{S}(G_0)^{\ell}}\, \le  \, k^{C}\max\{1,\Ex{Y_S(G_0)}^\ell\}
		\]
  for all $n$ sufficiently large.
	\end{lemma}

	We are now ready to prove Proposition~\ref{prop:rough}.

	\begin{proof}[Proof of Proposition~\ref{prop:rough} assuming Lemma~\ref{lem:moments}]
		Let $\ell=\lceil 8\eps^{-1}\rceil$ and let $S$ be a set of size at most three.  By Lemma~\ref{lem:moments} and Markov's inequality, we have 
		\begin{align*}
			\pr{Y_S(G_0)\, >\, n^{\eps}\max\{1,\Ex{Y_{S}(G_0)}}\, &=\, \pr{Y_{S}(G_0)^{\ell}\, >\, n^{\eps\ell}\max\{1,\Ex{Y_{S}(G_0)}^{\ell}\}
			}
			\\
			&\le\, \Ex{Y_{S}(G_0)^{\ell}}n^{-\eps\ell}\max\{1,\Ex{Y_{S}(G_0)}^{\ell}\}^{-1}\\
			&\le\,  \frac{k^{C}}{n^{8}}\\
   &\le\, n^{-7}\, .
		\end{align*}
		for all sufficiently large $n$.
        Above, we used that $k \le k_0 = O(\log n)$, as $p$ is uniformly bounded away from 1.
  
		By a union bound over at most $n^3$ choices of the set $S$, with $|S|\le 3$, the resulting probability is at most $n^{-4}$, for all sufficiently large $n$.
	\end{proof}

	We now turn to the task of proving Lemma~\ref{lem:moments}. For the rest of this section, we fix $\gamma_{-}>2$, $p_{+}\in (0,1)$ and $(p,\gamma,k) \in \Lambda_{\gamma_{-},p_{+}}$. Let $S$ be a fixed set of size $s$ and define $G_S=G_0 \cup\binom{S}{2}$.  Let $\K_S$ denote the set of all $k$-cliques in the complete graph $K_n$ which contain the set $S$.  We note that
	\[
	Y_{S}(G_0)\, =\, \sum_{K\in \K_S} 1_{K\subseteq G_S}\, 
	\]
	and
	\[
	Y_{S}(G_0)^{\ell}\, =\, \sum_{(K^1,\dots ,K^{\ell})\in \K_S^{\ell}} 1_{K^1\cup \dots\cup K^{\ell}\subseteq G_S},
	\]
	where $\K_S^{\ell}$ denotes the set of $\ell$-tuples of $\K_S$.
	Taking expectation, we see that the $\ell$th moment is given by 
	\[
	\Ex{Y_S(G_0)^{\ell}}\, =\, \sum_{(K^1,\dots ,K^{\ell})\in \K_S^{\ell}} \pr{K^1\cup \dots\cup K^{\ell}\subseteq G_S}\, .
	\]

	Even though the probability is independent of the ordering of the sequence of cliques $K^1,\dots ,K^{\ell}$, it is useful to just consider certain orders, as our proof will have a recursive argument which depends on properties of the sequence. 
	We say that a sequence of cliques $(K^1,\dots ,K^{\ell})$ is \emph{good} if it is ordered in such a way that the next clique always maximises the intersection with the union of the existing cliques in the sequence. That is, $(K^1,\dots ,K^{\ell})$ is good if for every $r \in [\ell-1]$ we have
	\[\left|V(K^{r+1})\cap \bigcup_{j\le r}V(K^j)\right|\, =\, \max_{i>r}\left|V(K^{i})\cap \bigcup_{j\le r}V(K^j)\right|.
	\]
	Every sequence of cliques may be reordered to be a good sequence: simply consider an algorithm that uses this criterion to select the next clique.

	Let $\K_S^{*,\ell}$ be the collection of good sequences in $\K_S^{\ell}$.
	As every good sequence corresponds to at most $\ell!$ sequences, we have
	\eq{must}
	\Ex{Y_S(G_0)^{\ell}}\, \le \, \ell! \sum_{(K^1,\dots ,K^{\ell})\in \K_S^{*,\ell}} \pr{K^1\cup \dots\cup K^{\ell}\subseteq G_S}.
	\eqe
	It is useful to classify a good sequence of $k$-cliques according to how ``clustered'' the cliques
	are. Before doing so, it is convenient to define $D_0$ to be the smallest integer such that $2^{32k\ell^2}\le n^{D_0}$ and $D$ to be the first integer which is at least $4\ell D_0+(1-p)^{-1}+s$.
	We emphasize that $D=\Theta(1)$. Moreover,
	as $(p,n,k) \in \Lambda_{\gamma_{-},p_{+}}$ and $\gamma_{-}>2$, we have $k < 2.1\log_{1/p_{+}} n$, and hence $D_0$ is at most $ 2.1 \cdot 32\ell^2 \log^{-1}(1/p_{+})$.

	It will be helpful to classify good sequences based on the intersection of the last clique with the vertices already present in previous cliques.
	
	\begin{defn}
		Let $(K^1,\dots ,K^{r+1})$ be a good sequence, and let
		\[
		I_{r+1}\, :=\, \Big|V(K^{r+1})\cap \bigcup_{j\le r}V(K^j)\Big|\, .
		\]
		We call $(K^1,\dots ,K^{r+1})$ \emph{small} if $I_{r+1}\le D$, \emph{middling} if $D<I_{r+1}<k-D$ and \emph{large} if $I_{r+1}\ge k-D$.
	\end{defn}
	
	Given a good sequence $(K^1,\dots ,K^{r+1})$, one can view its construction as proceeding through steps $(K^1), (K^1,K^2),\dots ,(K^1,\dots ,K^{r+1})$. It will turn out that if any of these steps is ``middling'' then we may win a large factor when bounding the expected count. 
	
	\begin{defn}
		We say that a good sequence $(K^1,\dots ,K^{r+1})$ is \emph{simple} if at all steps $(K^1,\dots ,K^{i})$ is small or large, i.e., there are no ``middling'' steps.
	\end{defn}
	
	We now state the lemma which bounds the contributions of steps of different types.
	
		\begin{restatable}{lemma}{onestep}\label{lem:onestep}
		Let $0\le r\le \ell$, let $(K^1,\dots ,K^r)$ be a good sequence and set $U = K^1 \cup \cdots \cup K^r$.
		Let $\C_g, \C_s, \C_m$, and $\C_l$ be the sets of $k$-cliques $K^{r+1}$ such that the sequence $(K^1,\dots ,K^r,K^{r+1})$ is categorized as good, small, middling or large, respectively.
		Then,
		\vspace*{2mm}
		\begin{enumerate}
			\item [$(i)$] $\sum_{K^{r+1} \in \C_g} \mathbb{P}\left( K^{r+1} \, \subseteq\, G_S\, \Big|\, U \, \subseteq\, G_S\, \right)\, \le\, (2e\ell)^k\max\{\Ex{Y_S(G_0)},1\}$;
			\vspace*{2mm}
			\item[$(ii)$] $\sum_{K^{r+1} \in \C_s} \mathbb{P}\left(K^{r+1} \, \subseteq\, G_S\, \Big|\, U \, \subseteq\, G_S\, \right) \, \le\, 2\Ex{Y_S(G_0)}$;
			\vspace*{2mm}
			\item[$(iii)$] $\sum_{K^{r+1} \in \C_m} \mathbb{P}\left(K^{r+1} \, \subseteq\, G_S\, \Big|\, U \, \subseteq\, G_S\, \right) \, \le\, n^{-D/16}\max\{\Ex{Y_S(G_0)},1\}\, .$
		\end{enumerate}
		Furthermore, if $(K^1,\dots ,K^r)$ is simple, then
		\[
		\sum_{K^{r+1}\in \C_l}\mathbb{P}\left(K^{r+1} \, \subseteq\, G_S\, \Big|\, U \, \subseteq\, G_S\, \right)\, \le\, 2(k\ell)^{4\ell D}\max\{\Ex{Y_S(G_0)},1\}\, .
		\]
	\end{restatable}
 
	As the proof of Lemma~\ref{lem:onestep} requires few interesting ideas and many technical steps, it may be found in Appendix~\ref{ap:one-step}.
	We may now deduce Lemma~\ref{lem:moments} assuming Lemma~\ref{lem:onestep}.  The proof relies on applying the one-step bounds of Lemma~\ref{lem:onestep} repeatedly.
	
	\begin{proof}[Proof of Lemma~\ref{lem:moments}]
        Since the case $\ell = 0$ is trivial, we may assume that $\ell \ge 1$.
		Let $r < \ell$ be a non-negative integer.
		Observe that the collection of good sequences of size $r$ can be split into two parts: the good simple sequences, which we shall denote by $\K_{\text{simple}}^{r}$ and the good non-simple sequences.
		By applying Lemma~\ref{lem:onestep}, we obtain that
		\vspace*{0.2cm}
		
		\noindent\resizebox{16.6cm}{0.8cm}{
			$\displaystyle\sum_{(K^1,\dots ,K^{r+1})\in \K_{\text{simple}}^{r+1}}  \pr{\bigcup_{i=1}^{r+1} K^i \subseteq G_S} \le 
			(k\ell)^{8\ell D}\max\{\Ex{Y_S(G_0)},1\} \sum_{(K^1,\dots ,K^{r})\in \K_{\text{simple}}^{r}}  \pr{\bigcup_{i=1}^{r} K^i\subseteq G_S} $}
		\vspace*{0.2cm}
		
		\noindent for all $r < \ell$. By applying the last inequality repeatedly, we obtain
		\begin{align}\label{eq:sum-simple}
			\sum_{(K^1,\dots ,K^{\ell})\in \K_{\text{simple}}^{\ell}} & \pr{\bigcup_{i=1}^{\ell} K^i \subseteq G_S} \le 
			(k\ell)^{8\ell^2 D}\big (\max\{\Ex{Y_S(G_0)},1\} \big)^{\ell}.
		\end{align}
		
		Recall that $\K_S^{*,r}$ denotes the collection of good sequences $(K^1,\ldots,K^r)$ of $k$-cliques containing $S$.
		To deal with the non-simple sequences, we first observe from~Lemma~\ref{lem:onestep} that for every pair of indices $0 \le r \le r' \le \ell$, we have
		\vspace*{0.2cm}
		
		\noindent\resizebox{16.5cm}{0.8cm}{
			$\displaystyle\sum\limits_{(K^1,\dots ,K^{r'})\in \K_S^{*,r'}} \pr{\bigcup_{i=1}^{r'} K^i \subseteq G_S} \le 
			\big((2e \ell)^{k}\max\{\Ex{Y_S(G_0)},1\}\big)^{r'-r} \sum_{(K^1,\dots ,K^{r})\in \K_{S}^{*,r}} \pr{\bigcup_{i=1}^{r} K^i\subseteq G_S}.$}
		\vspace*{0.2cm}
		
		\noindent For simplicity, set $\C_{NS} \coloneqq \K_S^{*,\ell} \setminus \K^{\ell}_{\text{simple}}$. 
        Note that $\C_{NS}$ comprises those sequences $(K^1,\ldots,K^{\ell})$ in $\K_S^{*,\ell}$ for which there exists an $r \in [\ell]$ such that the size of $K^{r} \cap (K^1\cup \cdots \cup K^{r-1})$ belongs to $\{D+1,\ldots,k-D-1\}$.
		Again by Lemma~\ref{lem:onestep} combined with our previous displayed inequality, it follows that
		\begin{align}\label{eq:sum-complex}
			\sum\limits_{(K^1,\dots ,K^{\ell})\in \C_{NS}} \pr{\bigcup_{i=1}^{\ell} K^i \subseteq G_S} \le 
			\big((2e \ell)^{k}\max\{\Ex{Y_S(G_0)},1\}\big)^{\ell} \cdot n^{-D/16}.
		\end{align}
		As $D$ was chosen so that $2^{32k\ell^2}\le n^{D}$, it follows from~\eqref{eq:sum-simple} combined with~\eqref{eq:sum-complex} that 
		\begin{align*}
			\sum\limits_{(K^1,\dots ,K^{\ell})\in \K_{S}^{*,\ell}} \pr{\bigcup_{i=1}^{\ell} K^i \subseteq G_S} \le 
			(k \ell)^{9\ell^2D}  \cdot \big(\max\{\Ex{Y_S(G_0)},1\}\big)^{\ell}.
		\end{align*}
		This together with~\eqr{must} finishes our proof.
	\end{proof}

	\subsection{More precise bounds for $Q(G_0)$ and $Y_e(G_0)$}\label{sec:Iprecise}
	In this section, we show that $Q(G_0)$ and $Y_e(G_0)$ are highly concentrated around $\tq(0)$ and $\ty(0)$, respectively.
	Let $\gamma_- >2$ and $p_+\in (0,1)$, and let $(p,n,k)\in \Lambda_{\gamma_-,p_+}$.
	First, recall that
	\[\tq(0)= \Ex{Q(G_0)} \qquad \text{and}\qquad \ty(0)= \binom{k}{2}\dfrac{\tq(0)}{e(G_0)}.\]
	As $e(G_0) = p\binom{n}{2} \pm (pn^2)^{3/4}$ with high probability and $\Ex{Y_e(G_0)} = \binom{n-2}{k-2}p^{\binom{k}{2}-1}$,
	we have
	\[\ty(0) =  \left(1 \pm n^{-1/4} \right)  \Ex{Y_e(G_0)}\]
	with high probability.
	Therefore, it suffices to show that 
	\begin{align}\label{eq:Q0andY0}
		Q(G_0) = (1\pm n^{-2\delta})\Ex{Q(G_0)} \qquad \text{and} \qquad Y_e(G_0) = (1\pm n^{-2\delta})\Ex{Y_e(G_0)}
	\end{align}
	for all $e \in K_n$ with high probability.
	
	We shall prove the upper tail part of~\eqref{eq:Q0andY0} using the deletion method of Janson and Ruci\'nski~\cite{JR1,JR2} and the lower tail part using Janson's inequality~\cite{Jan90}.  Note that the result for $Q(G_0)$ follows from the result for $Y_e(G_0)$ for all $e\in E(K_n)$, as $Q(G_0)=\binom{k}{2}^{-1}\sum_{e\in E(G_0)}Y_e(G_0)$.
	Thus, it suffices to prove the result for the $Y_e(G_m)$.
	
	We start with the upper tail, which we prove using the deletion method (Theorem~\ref{thm:deletion-method}).  
	Let $\mathcal{A}_e$ be the collection of $k$-cliques in $K_n$ containing the edge $e$. 
	We denote by $X_K^e$ the indicator random variable of the event that $K\subseteq G_0\cup \{e\}$, so that $Y_e(G_0)=\sum_{K \in \mathcal{A}_e}X_K^e$.  Now, let us make a few observations:
	\begin{enumerate}
		\item[$(i)$] Consider the dependency graph with vertex set $\mathcal{A}_e$ in which $K$ and $K'$ are adjacent if $|K\cap K'|\ge 3$.
		Observe that $X_K^e$ is independent of the family $\{X_{K'}^e: |K \cap K'| \le 2\}$. 
		\item[$(ii)$] The random variables $\tilde{X}_{\alpha}$ defined in Theorem~\ref{thm:deletion-method} correspond in our context to
		\[
		\tilde{X}_{K}^e\, :=\, \sum_{\substack{K' \in \mathcal{A}_e: \\|K\cap K'|\ge 3}}X_K^e\, \le\, \sum_{S\subseteq K:|S|=3}Y_{S}(G_0).
		\]
		\item[$(iii)$] The analog of the random variable $X^{*}$ in our context,  $X^{*}_e = \max_{K \in \mathcal{A}_e} \tilde{X}_{K}^e$, is at most $\binom{k}{3}\max_{S:|S|=3}Y_S(G_0)$.  
		Let $S$ be an arbitrary set of size three. Note that
		\[\Ex{Y_S(G_0)} = \binom{n-3}{k-3}p^{\binom{k}{2}-3} = n^{\gamma-3+o(1)}.\]
		It follows from Proposition~\ref{prop:rough} that
		\begin{align*}
			X^{*}_e\le \binom{k}{3}n^{\delta}\max\{1,\Ex{Y_{S}(G_0)}\} = n^{\delta+o(1)+\max \{0,\gamma-3\}}
		\end{align*} except with probability at most $n^{-4}$ for all sufficiently large $n$.
		\item[$(iv)$] Note that
		\begin{align*}
			\Ex{Y_e(G_0)} = \binom{n-2}{k-2}p^{\binom{k}{2}-2} = n^{\gamma-2+o(1)}.
		\end{align*}
		As $6\delta < \min\{\gamma-2,1\}$ (recall the definition of $\delta$ in~\eqref{eq:delta}), it follows from $(iii)$ that $X^{*}_e \le \Ex{Y_e(G_0)}/n^{5\delta}$ except with probability at most $n^{-4}$ for all sufficiently large $n$.
	\end{enumerate}
	We now apply the deletion method (Theorem~\ref{thm:deletion-method}) with $\mu=\Ex{Y_e(G_0)}$, $t=n^{-2\delta}\Ex{Y_e(G_0)}$ and $r=n^{3\delta}/2$ to obtain
	\[
	\pr{Y_e(G_0)\ge \left(1+\dfrac{1}{n^{2\delta}}\right)\Ex{Y_e(G_0)}}\, \le\, \left(1+\dfrac{1}{2n^{2\delta}}\right)^{-n^{3\delta}/2}\, +\, n^{-4}\, \le\, n^{-3}\, ,
	\]
	for all sufficiently large $n$. Finally, applying an union bound over $\binom{n}{2}$ choices of $e \in K_n$, there is probability at most $n^{-1}$ that this event occurs for some $e \in E(K_n)$.
	
	We now prove the lower tail part of these results using Janson's inequality (Theorem~\ref{thm:janson}).  The context is the same as above, with $Y_e(G_0)=\sum_{K \in \mathcal{A}_e}X_K^e$, where $\mathcal{A}_e$ denotes the collection of $k$-cliques in $K_n$ containing the edge $e$.
	We observe that $X_K^e$ is exactly the indicator of an event that a particular set of $\binom{k}{2}-1$ edges are in $G_0$.
	As before, we also consider the dependence graph on the set of $k$-cliques in $\mathcal{A}_e$, where $K$ and $K'$ are adjancent if $|K \cap K'| \ge 3$.
	In order to apply Janson's inequality, the first step is to calculate $\bar{\Delta}$, which in this case is
	\begin{align}\label{eq:delta-bar}
		\bar{\Delta}\, &=\, \sum_{\substack{(K,K')\in \mathcal{A}_e^2:\\|K \cap K'| \ge 3}} p^{e(K\cup K')} \nonumber\\
		& = \,  \binom{n}{k-2}p^{\binom{k}{2}-1} \sum \limits_{i=1}^{k-2}\binom{k-2}{i}\binom{n-k}{k-2-i} p^{\binom{k}{2}-\binom{i+2}{2}}.
	\end{align}
	Let us denote by $a_i$ the $i$-th term of the sum above.
	\begin{claim}\label{claim:boundonai}
		$a_i \le n^{o(1)}\max\{a_1,a_{k-2}\}$ for all $i \in \{1,\ldots,k-2\}$.
	\end{claim}
	
	\begin{proof}
		Observe that
		\begin{align}\label{eq:ratioais}
			\dfrac{a_{i+1}}{a_i} = \dfrac{(k-2-i)^2}{i+1} \cdot \dfrac{1}{n-2k+i+3} \cdot p^{-i-2},
		\end{align}
		and hence the ratio of two consecutive ratios, say $a_{i+1}/a_i$ over $a_i/a_{i-1}$ is
		\begin{align*}
			\dfrac{a_{i+1}}{a_i} \cdot \left(\dfrac{a_i}{a_{i-1}}\right)^{-1} = \left(1-\dfrac{1}{k-i-1}\right)^2 \cdot \left(1-\dfrac{1}{i+1}\right) \cdot \left(1-\dfrac{1}{n-2k+i+3}\right) \cdot p^{-1}.
		\end{align*}
		For simplicity, set $D:=(1-p_{+}^{1/4})^{-1}+1$. As $p < p_{+}$, for every $i \ge D$ we have that
		\begin{align}\label{eq:ratioais-3}
			\dfrac{a_{i+1}}{a_i} \cdot \left(\dfrac{a_i}{a_{i-1}}\right)^{-1} \ge \left(1-\dfrac{1}{k-i-1}\right)^2 \cdot p_{+}^{-1/2}
		\end{align}
		for all $n$ sufficiently large. 
		It follows from~\eqref{eq:ratioais-3} that the ratio $(a_{i+1}/a_i)/(a_i/a_{i-1})$ is at least 1 for all $i \in \{D,\ldots,k-D\}$.
		As $a_{D+1}/a_D = n^{-1+o(1)}<1$ (here we used~\eqref{eq:ratioais}, together with the facts that $D$ is constant, $p = n^{o(1)}$, and $k = O(\log n)$) and the sequence of ratios $(a_{i+1}/a_i)_{i=D}^{k-D}$ is increasing, there must exist an index $j$ such that $(a_i)_{i=D}^j$ is non-increasing and $(a_i)_{i=j}^{k-D}$ is  non-decreasing.
		It follows from this that
		\eq{midai}
		a_i\,\le \, \max\{a_D,a_{k-D}\} \qquad \text{for all} \, D\le i \le k-D.
		\eqe
		Now, it follows from~\eqref{eq:ratioais}  that $a_{i+1}/a_i = n^{-1+o(1)}$ for all $i \in \{1,\ldots,D\}$, and hence $a_i \le a_1$ for all $i \in \{1,\ldots,D\}$.
		Similarly, as $\log_{1/p}n \le k = O(\log n)$, it follows from~\eqref{eq:ratioais} that for all $i \in \{k-D,\ldots,k-3\}$ we have 
        \begin{align*}
            \dfrac{a_{i+1}}{a_i} \ge \dfrac{1}{i+1} \cdot \dfrac{1}{n} \cdot p^{-i-2} \ge \dfrac{1}{kn} \cdot  p^{-k+D-2} \ge k^{-1}p^{D-2}.
        \end{align*}
        Hence,
		$a_i \le n^{o(1)}a_{k-2}$ for all  $i \in \{k-D,\ldots,k-2\}$.
		This together with~\eqr{midai} proves our claim.
	\end{proof}
	For simplicity, set $\mu=\Ex{Y_e(G_0)}$. From~\eqref{eq:delta-bar} and Claim~\ref{claim:boundonai}, we obtain $\bar{\Delta} = n^{o(1)} \max \{\mu, \mu^2n^{-1} \}$.
	We may now apply Janson's inequality (Theorem~\ref{thm:janson}) with $t=n^{-2\delta}\mu$ to obtain
	\[
	\pr{Y_e(G_0)\le \left(1-\dfrac{1}{n^{2\delta}}\right)\Ex{Y_e(G_0)}}\, \le\, \exp\left(\frac{-n^{-4\delta}\mu^2}{n^{o(1)} \max \{\mu, \mu^2n^{-1} \}}\right)\, .
	\]
	If the first term in the denominator is the maximum, then we obtain a bound of the form $\exp(-n^{\gamma-2-4\delta-o(1)})\le \exp(-n^{\delta})\le n^{-3}$ for all sufficiently large $n$ (recall from~\eqref{eq:delta} that $\delta = \min\{\gamma-2,1\}/10$).  If the second term is the maximum, then we obtain $\exp(-n^{1-4\delta+o(1)})\le n^{-3}$ for all sufficiently large $n$.  In either case, applying an union bound over $\binom{n}{2}$ choices of $e \in K_n$, there is probability at most $n^{-1}$ that this event occurs for some $e \in E(K_n)$.

	\section{Controlling the evolution of \texorpdfstring{$Q(G_m)$}{Lg}}\label{sec:Q}
	
	We recall the definitions of $m_{*}$, $\tau_Y$, $\tau_Q$ and $\tau$, see~\eqref{eq:defn-m-star} and~\eqref{eq:stopping-times} in Section~\ref{sec:over}.
	In this section, our goal is to prove Proposition~\ref{prop:Q}, which states that $\pr{0<\tau=\tau_Q<m_{*}}\, \to\, 0$.
	In other words, we show that it is unlikely that the process $(Q(G_m))_{m\ge 0}$ goes wrong before time $m_{*}$, and is the \emph{first} to do so.
	
	It is useful to define stopping times depending on whether the eventual failure occurs with $Q(G_m)$ ``too large'' (greater than $(1+g_Q(m))\tq(m)$), or ``too small'' (less than $(1-g_Q(m))\tq(m)$).
	Define $\tau_Q^{+}$ to be the minimum of $m_{*}$ and the first value of $m$ such that $\tq(m)>(1+g_Q(m))\tq(m)$.
	Similarly, define $\tau_Q^{-}$ to be the minimum of $m_{*}$ and the first value of $m$ such that $\tq(m)<(1-g_Q(m))\tq(m)$.
	Clearly, we have $\tau_Q = \min \{\tau_Q^{-}, \tau_Q^{+}\}$, and hence we need to show that both probabilities $\pr{0<\tau=\tau_Q^{+}<m_{*}}$ and $\pr{0<\tau=\tau_Q^{-}<m_{*}}$ tend to zero as $n$ grows.
	We focus on the case ``too large'', that is, for the rest of this section our main goal in to prove that 
	\[\pr{0<\tau=\tau_Q^{+}<m_{*}}\, \to\, 0.\]
	The proof for ``too small'' is essentially identical. In this case, instead of defining the sequence $X_i$ in terms of $Q(G_{i})\, -\, (1+g_Q(i))\tq(i)$, one would use instead $(1-g_Q(i))\tq(i)\, -\, Q(G_{i})$.
	
	It is natural at this point to encode the event $\{0<\tau=\tau_Q^{+}<m_{*}\}$ as a deviation event of a random process. Recall from Section~\ref{sec:over} that we have defined the random process $(X_m)_{m\ge0}$ as
	\[
	X_i\, :=\, \begin{cases} Q(G_{i})\, -\, (1+g_Q(i))\tq(i),& \text{ if } i\le \tau\\  
		Q(G_{\tau})\, -\, (1+g_Q(\tau))\tq(\tau),& \text{ if } i> \tau\,  .\end{cases}
	\]
	We observe that for the event $\{0<\tau=\tau_Q^{+}<m_{*}\}$ to occur it is necessary that the following holds:

	\begin{enumerate}
		\item[(i)] As $\tau > 0$, there is no problem with the behaviour of the random variables in the original random graph $G_0\sim G(n,p)$, and hence all items in Proposition~\ref{prop:refined} are satisfied.
		In particular, $Q(G_0)=(1\pm n^{-\delta})\tq(0)$ and, since $g_Q(0)=2n^{-\delta}$, we have $X_0 \le -g_Q(0)\tq(0)/2$;
		\item[(ii)] As $\tau_Q^+ < m_{*}$, there exists $m^{\dagger}\in \{1,\dots, m_{*}-1\}$ such that the failure event $Q(G_{m^{\dagger}}) > (1 + g_Q(m))\tq(m^{\dagger})$ occurs at time $m^{\dagger}$, and hence $X_{m^{\dagger}}>0$;
		\item[(iii)] As $\tau = \tau_Q^+$, no problem has occurred previously, so that $Q(G_m) = (1\pm g_Q(m))\tq(m)$ and $Y_e(G_m) = (1\pm g_Y(m))\ty(m)$ for all $e\in E(K_n)$ and all $0 \le m < m^{\dagger}$.
	\end{enumerate}
	
	It follows from items (i) and (ii) that if $\{0<\tau=\tau_Q^{+}<m_{*}\}$ occurs, then $X_{m^{\dagger}}-X_0 > g_Q(0)\tq(0)/2$ for some $m^{\dagger}\in \{1,\dots, m_{*}-1\}$. 
	If we show that the process $(X_m)_{m\ge 0}$ is a supermartigale, then we can bound the probability of the latter event via Azuma's inequality (Lemma~\ref{lem:hoeff-azuma}).
	Actually, we show something slightly different: instead of showing that the whole process is a supermartingale, we show that $X_m$ behaves as a supermartingale whenever $Q(G_m)$ is close to exceeding the allowed upper bound, $(1+g_Q(m))\tq(m)$.  This restriction is necessary, see~\eqref{eq:Exdiff} in the proof of Lemma~\ref{lem:Qsm}.
	
	For $i \in \{0,\dots, m_{*}\}$, define 
	$\psi_{i}$ to be the smallest index $m \in \{i,\ldots,m_{*}\}$ such that 
	\[ Q(G_m) \le  \left(1 + \dfrac{g_Q(m)}{2}\right)\tq(m).\] 
	If such an index does not exist, we set $\psi_{i} = m_{*}$.
	The key ingredients to prove Proposition~\ref{prop:Q} are the next two lemmas.
	
	\begin{lemma}\label{lem:Qsm} 
		Let $i \in \{0,1,\dots, m_{*}\}$.
		The process $(X_m: i \le m \le \psi_i)$ is a supermartingale with respect to its natural filtration, provided $n$ is sufficiently large.\end{lemma}

	\begin{lemma}\label{lem:Qchange} 
		For every $m \in \{0,\ldots,m_{*}-1\}$ we have
		\[|X_{m+1}-X_{m}| = O\left(\dfrac{k^4\tq(m)}{pn^2}\right).\]
	\end{lemma}

	\begin{proof}[Proof of Proposition~\ref{prop:Q} assuming Lemmas~\ref{lem:Qsm} and~\ref{lem:Qchange}]
		As we discussed at the beginning of this section, we focus on showing that $\pr{0<\tau=\tau_Q^{+}<m_{*}}\, \to\, 0$. The proof that $\pr{0<\tau=\tau_Q^{-}<m_{*}}\, \to\, 0$ is essentially identical.
		
		If the event $ \{0<\tau=\tau_Q^{+}<m_{*}\}$ occurs, then $G_0$ satisfies all items (1)--(4) in Proposition~\ref{prop:I}.
		Moreover, since $g_Q(0)=2n^{-\delta}$, we have $Q(G_0) = (1\pm g_Q(0)/2)\tq(0)$.
		This implies that there exists a step $i \in \{1,\ldots,\tau\}$ for which
		$Q(G_{i-1}) \le (1 + g_Q(i-1)/2)\tq(i-1)$ and $Q(G_{m}) > (1 + g_Q(m)/2)\tq(m)$ for all $m \in \{i,\ldots,\tau\}$.
		In particular, we have $\psi_i = m_{*}$ and, as the process freezes at the stopping time $\tau$, by Lemma~\ref{lem:Qsm} the process  $(X_m: m \ge i)$ is a supermartigale.

		As we expect $Q(G_i)$ and $Q(G_{i-1})$ to be ``very close'' to each other, we have an upper bound on $Q(G_i)$ which is better than $Q(G_i) \le (1 + g_Q(i))\tq(i)$.
		Note that
		\begin{align*}
			Q(G_{i}) \le Q(G_{i-1}) < \left(1 + \dfrac{g_Q(i-1)}{2}\right)\tq(i-1) < \left(1 + \dfrac{3g_Q(i)}{4}\right)\tq(i).
		\end{align*}
		The last inequality follows from the fact that $g_Q(i) = g_Q(i-1)\big(1+n^{-2+o(1)}\big)$ and that $\tq(i)=\tq(i-1)\big(1-n^{-2+o(1)}\big)$.
		In particular, we have $X_i \le -g_Q(i)\tq(i)/4$.
		As $X_{m_{*}}=X_{\tau}>0$, we obtain $X_{m_{*}}-X_i \ge g_Q(i)\tq(i)/4$.
		
		For each $i \in \{1,\ldots,m_{*}\}$, define
		\begin{align*}
			A_i :=	\left \{ X_{m_{*}}-X_i>\dfrac{g_Q(i)\tq(i)}{4}\right \} \cap \left \{ \psi_i = m_{*} \right \}.
		\end{align*}
		From all the discussion above, it follows that
		\begin{align}\label{eq:union-Ai}
			\{0<\tau=\tau_Q^{+}<m_{*}\} \se \bigcup  _{i=1}^{m_{*}}	A_i.
		\end{align}	
		We now fix some $i \in [m_{*}]$ and bound the probability of the event $A_i$.
		We encode this event as a deviation event of a supermartingale.
		By Lemmas~\ref{lem:Qsm} and~\ref{lem:Qchange}, we have that $(X_m)_{m= i}^{\psi_i}$ is a supermartingale with increments bounded by $O(1)k^4\tq(i)/pn^2$.
		As $ m_{*} \le pn^2/4$, by the Hoeffding--Azuma inequality (Lemma~\ref{lem:hoeff-azuma}) we obtain
		\begin{align}\label{eq:bound-Ai}
			\pr{A_i}\, \le\,  \exp\left(\frac{-g_Q(i)^2\tq(i)^2}{O(1)pn^2\big(k^4\tq(i)/pn^2\big)^2}\right) \, \le \, \exp\left(\frac{-g_Q(i)^2p n^{2}}{O(1)k^8}\right) \, \le \, n^{-3}
		\end{align}
		for all sufficiently large $n$.
		Finally, we may take a union over the at most $n^2$ values of $i$. It follows from Proposition~\ref{prop:I},~\eqref{eq:union-Ai} and~\eqref{eq:bound-Ai} that $\pr{0<\tau=\tau_Q^{+}<m_{*}}$ tends to $0$ as $n$ tends to infinity.
	\end{proof}

	In order to prove Lemmas~\ref{lem:Qsm} and~\ref{lem:Qchange}, we need a good approximation to the increment coming from the deterministic part of $(X_m)_{m\ge0}$, namely
	\begin{align}\label{eq:Rm}
		R_m \coloneqq (1+g_Q(m+1))\tq(m+1)\, -\, (1+g_Q(m))\tq(m).
	\end{align}
	\begin{claim}\label{claim:change-Rm}
		For every $m \in \{0,\ldots,m_{*}-1\}$ we have
		\[R_m = - \frac{\binom{k}{2}^2 \tq(m)}{e(G_{m})} \pm \dfrac{k^8\tq(m)}{p^2n^4}.\]
	\end{claim}
	
	\begin{proof}
		Recall the definitions of $\tq$ and $g_Q$ in Section~\ref{sec:over}, see~\eqref{def:tq} and~\eqref{eq:gQerror}. As $e(G_{m})=e(G_0)-m\binom{k}{2}$, we have
		\[\tq(m+1)=\left(1-\dfrac{\binom{k}{2}^2}{e(G_m)}\right) \tq(m) \quad \text{and} \quad g_Q(m+1) = \left(1+ \dfrac{\binom{k}{2}^2}{e(G_m)}\right)g_Q(m),\]
		and hence we obtain
		\begin{align*}
			\nonumber R_m &= \left(1\, +\, g_Q(m)\, +\, \frac{\binom{k}{2}^2}{e(G_{m})} g_Q(m)\right)\left(1-\frac{\binom{k}{2}^2}{e(G_{m})}\right)\tq(m)\, -\, (1+g_Q(m))\tq(m)\nonumber \\
			& = - \frac{\binom{k}{2}^2 \tq(m)}{e(G_{m})} - \frac{\binom{k}{2}^4 g_Q(m)\tq(m)}{e(G_{m})^2}\\
			& = - \frac{\binom{k}{2}^2 \tq(m)}{e(G_{m})} \pm \dfrac{k^8\tq(m)}{p^2n^4}.
		\end{align*}
	\end{proof}

	\begin{proof}[Proof of Lemma~\ref{lem:Qsm}] 
		We observe that the result holds trivially if either $\psi_i = i$ or $\tau = 0$. We therefore assume that $\psi_i > i$ and $\tau > 0$ throughout the proof.
	
		Fix $m \in \{i,\ldots, \psi_i-1\}$. 
		As the increment $X_{m+1}-X_{m}$ is identically $0$ if $m \ge \tau$, we may assume that $m < \tau$ and so the increment $X_{m+1}-X_{m}$ is given by
		\[
		Q(G_{m+1})\, -\, Q(G_{m})\, -\,  R_m.
		\]
		Our aim is to prove that $\Ex{X_{m+1}-X_{m}|G_{m}}\le 0$,
		which is equivalent to
		\eq{suffQ}
		\mathbb{E}\Big[Q(G_{m+1})\, -\, Q(G_{m})\Big| G_{m}\Big]\, -\,  R_m\, \le 0\, .
		\eqe
		
		Let us first analyse the one-step change $Q(G_{m+1})-Q(G_{m})$.  At step $m$ of the process a uniformly random clique $K\in G_{m}$ is selected and its edges are removed.  The $k$-cliques that are destroyed in $G_m$ are precisely those that share an edge with $K$.  It follows that the number of $k$-cliques removed in total is at most
		\[
		\sum_{e\in E(K)}Y_e(G_{m})\, .
		\]
		This might be a slight overcount, as cliques that have intersection at least $3$ with $K$ may be counted more than once.  
		As $\tau >0$ (so that $G_0$ satisfies all items (1)--(4) in Proposition~\ref{prop:I}), the overcount is by at most $\sum_{S\subseteq K:|S|=3}Y_S(G_{m})\le k^3 n^{\delta}\max\{1,\Ex{Y_{S}(G_0)}\}$.  
		It follows that
		\eq{Qchange}
		Q(G_{m+1})\, -\, Q(G_{m})\, =\, -\sum_{e\in E(K)}Y_e(G_{m})\, \pm\, E_1\, ,
		\eqe
		where $E_1:=k^3n^{\delta}\max\{1,\Ex{Y_{S}(G_0)}\}$.
		
		Now, define
		\[
		\widebar{Y}(G_{m})\, :=\, \frac{1}{e(G_{m})}\sum_{e\in E(G_{m})}Y_e(G_{m})\, =\, \frac{\binom{k}{2}Q(G_{m})}{e(G_{m})}.
		\]
		The quantity $\widebar{Y}(G_{m})$ is the average $Y_e(G_{m})$ value over $e\in G_{m}$.  We may now write each $Y_e(G_{m})$ as $(1+\eta_e)\widebar{Y}(G_{m})$ where $\sum_{e}\eta_e=0$.  
		As $m < \tau$, we have
		\begin{align}\label{eq:tyandyavg}
			\ty(m) = (1\pm g_Q(m)) \widebar{Y}(G_m) \qquad
			\text{and} \qquad Y_e(G_m) = (1\pm g_Y(m)) \ty(m)
		\end{align}
		for all $e \in E(K_n)$.
		Moreover, it follows from~\eqref{eq:tyandyavg} that
		\[(1+\eta_e)\widebar{Y}(G_{m}) =  Y_e(G_m) = (1\pm g_Y(m)) (1\pm g_Q(m)) \widebar{Y}(G_m),\]
		and hence
		\[|\eta_e| \le g_Q(m)+g_Y(m)+(g_Qg_Y)(m) \le 2g_Y(m).\]
		for all $e \in E(K_n)$. The last inequality follows from~\eqref{eq:gQerror},~\eqref{eq:gY} and~\eqref{eq:boundongQgY}.
		
		By taking the conditional expectation in~\eqr{Qchange}, we obtain
		\begin{align*}
			\mathbb{E}\Big[Q(G_{m+1})\, -\, Q(G_{m})\Big| G_{m}\Big]\, & =\, \frac{-1}{Q(G_{m})}\sum_{K}\sum_{e\in E(K)}Y_e(G_{m}) \, \pm\, E_1\phantom{\bigg|}\\
			& =\, \frac{-1}{Q(G_{m})}\sum_{e\in E(G_{m})}Y_e(G_{m})^2\, \pm \, E_1\phantom{\bigg|}.
		\end{align*}
		In the first equality, the first sum is over all $k$-cliques $K$ in $G_m$. The second equality follows by changing the order of summation and noticing that $Y_e(G_m)$ shows up $Y_e(G_m)$ times in the double sum.
		By replacing $Y_e(G_{m})$ by $(1+\eta_e)\widebar{Y}(G_{m})$ and using that $\sum_{e}\eta_e=0$ and that $|\eta_e| \le 2g_Y(m)$ for all $e \in E(G_m)$, we obtain
		\begin{align}\label{eq:expdeltaQ}
			\nonumber\mathbb{E}\Big[Q(G_{m+1})\, -\, Q(G_{m})\Big| G_{m}\Big]\, & =\, \frac{-1}{Q(G_{m})}\sum_{e\in E(G_{m})}(1+\eta_e)^2\widebar{Y}(G_{m})^2\, \pm E_1 \phantom{\bigg|}\\
			& =\, -\frac{\binom{k}{2}^2 Q(G_{m})}{e(G_{m})}\, \pm \, E_1\, \pm\, E_2 \phantom{\bigg|},
		\end{align}
		where $E_2:=k^4g_{Y}(m)^2 Q(G_{m})/e(G_{m})$.

		The one-step change coming from the determinist part in the left-hand side of~\eqr{suffQ} was analysed in Claim~\ref{claim:change-Rm}, which gives
		\begin{align}\label{eq:deltatildaQ}
			R_m = - \frac{\binom{k}{2}^2 \tq(m)}{e(G_{m})} \pm E_0,
		\end{align}
		where $E_0:= k^8\tq(m)/(p^2n^4)$.
		Now, set $E:=E_0+E_1+E_2$.
		
		\begin{claim}\label{eq:boundonE}
			$E \ll g_Q(m)\tq(m)/e(G_m).$
		\end{claim}
		
		\begin{proof}
			
			Note that
			\[\dfrac{k^4 g_Y(m)^2}{g_Q(m)} = n^{-\delta+o(1)} \exp \left( (3+o(1))\sum \limits_{i=1}^{m-1} \dfrac{\binom{k}{2}^2}{e(G_i)}\right) \le n^{-\delta/4+o(1)}.\]
			Moreover, since $n^{-\delta} \le g_Q(m)$ (see~\eqref{eq:boundongQgY}), it follows that $E_0+E_2\ll g_Q(m)\tq(m)/e(G_{m})$.
			In order to upper bound $E_1$, we first lower bound $g_Q(m)\tq(m)$ in terms of $\tq(0)$.
			Recall the definitions of $\tq$ and $g_Q$ in Section~\ref{sec:over}, see~\eqref{def:tq} and~\eqref{eq:gQerror}. 
			Note that
			\[g_Q(m)\tq(m)=2n^{-\delta}\prod_{i=0}^{m-1}\left(1-\dfrac{\binom{k}{2}^4}{e(G_i)^2}\right)\tq(0) \ge n^{-\delta}\exp \left(- \dfrac{k^8}{p^2n^4} \cdot m\right)\tq(0)\ge n^{-\delta+o(1)}\tq(0).\]
			In the second inequality, we used that $1-x/2 \ge e^{-x}$ for all $x \in (0,1/2)$ and that $e(G_i)\ge pn^2/4$ for all $i \le m_{*}$; in the third inequality, we used that $m \le m_{*}$.
			As $e(G_m) \le e(G_0)$, we have
			\begin{align}\label{eq:gqe}
				\dfrac{g_Q(m)\tq(m)}{e(G_m)}\ge n^{-\delta+o(1)}\dfrac{\tq(0)}{e(G_0)} = n^{\gamma-2-\delta+o(1)}.
			\end{align}
			On the other hand, $E_1 \le n^{\delta+o(1)+\max\{0,\gamma-3\}}.$
			Recall from~\eqref{eq:delta} that $\delta =\min\{\gamma-2,1\}/10.$
			Thus, it follows from~\eqref{eq:gqe} that $E_1 \ll g_Q(m)\tq(m)/e(G_m)$. 
			This finishes the proof of the claim.
		\end{proof}

		We are now nearly ready to complete our proof of~\eqr{suffQ}.  
		As $m \in \{i,\ldots,\psi_i-1\}$, we have $Q(G_{m})>(1+g_Q(m)/2)\tq(m)$. By combining this with~\eqref{eq:expdeltaQ},~\eqref{eq:deltatildaQ} and Claim~\ref{eq:boundonE}, we obtain
		\begin{align}
			\Ex{X_{m+1}-X_{m}|G_{m}}&=\, -\frac{\binom{k}{2}^2 (Q(G_{m})-\tq(m))}{e(G_{m})}\, \pm \, E\phantom{\bigg|}\label{eq:Exdiff} \\
			& \le\, -\frac{\binom{k}{2}^2 g_Q(m)\tq(m)}{2e(G_{m})}\, \pm \, E \phantom{\bigg|} \nonumber \\
			& \le \, 0\, \phantom{\bigg|} \nonumber.
		\end{align}
		We remark that this is why we need to use the bound $Q(G_m)\ge(1+g_Q(m)/2)\tq(m)$.
		If we only use the crude bound $Q(G_m)\ge(1-g_Q(m))\tq(m)$, this would not be enough to show that $(X_m: i \le m \le \psi_i)$ is a supermartingale.
        \end{proof}

	We now prove Lemma~\ref{lem:Qchange}, on the absolute value of the increment $|X_{m+1}-X_{m}|$.
	
	\begin{proof}[Proof of Lemma~\ref{lem:Qchange}] 
		If $m \ge \tau$, then the increment is identically $0$, and so we assume $m \le \tau-1$.
		The absolute value $|X_{m+1}-X_{m}|$ is at most the maximum of
		\[
		\Big|Q(G_{m+1})\, -\, Q(G_{m})\Big|\quad \text{and}\quad  \Big|(1+g_Q(m+1))\tq(m+1)\, -\, (1+g_Q(m))\tq(m)\Big|\, .
		\]
		Then, it suffices to prove that each of these is at most $O(1)k^4\tq(m)/pn^2$.  
		By Claim~\ref{claim:change-Rm}, the second expression is equal to $\binom{k}{2}^2 \tq(m)/e(G_{m})\, \pm\, k^8\tq(m)/(p^2n^4)=O(1)k^4\tq(m)/pn^2$.  
		The first expression, $\big|Q(G_{m+1})\, -\, Q(G_{m})\big|$, is exactly the number of $k$-cliques removed from $G_{m}$ when we remove the randomly selected $k$-clique $K$.  
		As $Y_e(G_m)\le (1+g_Y(m))\ty(m) \le 2\ty(m)$ for every $e \in K_n$ and $m < \tau$,
		for every choice of $K$ we obtain
		\[
		\Big|Q(G_{m+1})\, -\, Q(G_{m})\Big|\, \le \,\sum_{e\in E(K)}Y_e(G_{m})\, \le\, 2\binom{k}{2}\ty(m) \, =\, \frac{O(1)k^4\tq(m)}{pn^2}\, ,
		\]
		as required.
	\end{proof}
	
	\section{Controlling the evolution of \texorpdfstring{$Y_e(G_m)$}{Lg} for all \texorpdfstring{$e\in E(K_n)$}{Lg}}\label{sec:Y}
	
	As in the previous section our objective is the prove that random variables associated with the process stay within their allowed intervals.  We recall that $\tau_{Y_e}$ is the stopping time associated with a random variable $Y_e(G_m)$ leaving $(1\pm g_Y(m))\ty(m)$ and $\tau_Y$ is the minimum of these values.  
	As in Section~\ref{sec:Q}, it is useful to define stopping times depending on whether the eventual failure occurs with $Y_e(G_m)$ ``too large'' (greater than $(1+g_Y(m))\ty(m)$), or ``too small'' (less than $(1-g_Y(m))\ty(m)$).
	For each $e \in E(K_n)$, define $\tau_{Y_e}^{+}$ to be the minimum of $m_{*}$ and the first value of $m$ such that $Y_{e}(m)>(1+g_Q(m))\ty(m)$.
	Similarly, define $\tau_{Y_e}^{-}$ to be the minimum of $m_{*}$ and the first value of $m$ such that $Y_e(m)<(1-g_Y(m))\ty(m)$.
	Lastly, we set
	\[\tau^{+}_Y = \min \{\tau_{Y_e}^{+}: e \in E(K_n)\} \qquad \text{and} \qquad \tau^{-}_Y = \min \{\tau_{Y_e}^{-}: e \in E(K_n)\}.\]

	Clearly, we have that $\tau_Y = \min \{\tau^{-}_Y, \tau^{+}_Y \}$, and hence we need to show that both probabilities $\pr{0<\tau=\tau_Y^{+}<m_{*}}$ and $\pr{0<\tau=\tau_Y^{-}<m_{*}}$ tend to zero as $n$ grows.
	As in Section~\ref{sec:Q}, we focus on the case where $Y_e(m)$ leaves the interval by being ``too large''.
	That is, for the rest of this section our main goal in to prove that $\pr{0<\tau=\tau_Y^{+}<m_{*}}\, \to\, 0$.
	For that, it suffices to prove that
	\[
	\pr{0<\tau=\tau_{Y_e}^{+}<m_{*}}\, \le\, n^{-3}
	\]
	for every $e \in E(K_n)$ and every sufficiently large $n$.
	The proof for the ``too small'' case is essentially identical. We omit the details.
	
	As in Section~\ref{sec:Q}, it is natural to encode the event $\{0<\tau=\tau_{Y_e}^{+}<m_{*}\}$ as a deviation event of a random process. Define
	\[
	X'_m\, :=\, \begin{cases} Y_e(G_{m})\, -\, (1+g_Y(m))\ty(m),& \text{ if } m \le \tau\\  
		Y_e(G_{\tau})\, -\, (1+g_Y(\tau))\ty(\tau),& \text{ if } m > \tau\,  .\end{cases}
	\]
	For $i \in \{0,\dots, m_{*}\}$, define $\psi'_{i}$ to be the smallest index $m \in \{i,\ldots,m_{*}\}$ such that 
	\[Y_e(G_m) \le  \left(1 + \dfrac{g_Y(m)}{2}\right)\ty(m).\] 
	If such an index does not exist, we set $\psi'_{i} = m_{*}$.
	The key ingredients to prove Proposition~\ref{prop:Y}  are the next two lemmas.

	\begin{lemma}\label{lem:Ysm}
		Let $i \in \{0,1,\dots, m_{*}\}$. The process $(X'_m: i \le m \le \psi'_i)$ is a supermartingale with respect to its natural filtration, provided $n$ is sufficiently large.\end{lemma}

	\begin{lemma}\label{lem:Ychange} 
		For every $m \in \{0,\ldots,m_{*}-1\}$, the increments $X'_{m+1}-X'_{m}$ satisfy the following properties:
		\begin{enumerate}
			\item[$(i)$] $|X'_{m+1}-X'_{m}|\, =\, O(1)k^2n^{\delta}\max\{1,\Ex{Y_S(G_0)}\}$, where $|S|=3$; 
			\vspace*{2mm}
			\item[$(ii)$] 
			$\Ex{|X'_{m+1}-X'_{m}|\big|G_{m}}\, =\, O\left(\dfrac{k^4\ty(m)}{pn^2}\right)$.
		\end{enumerate}
	\end{lemma}

	\begin{proof}[Proof of Proposition~\ref{prop:Y}, assuming Lemmas~\ref{lem:Ysm} and~\ref{lem:Ychange}]
		
		As we discussed at the beginning of this section, we focus on showing that  $\pr{0<\tau=\tau_{Y_e}^{+}<m_{*}}\le n^{-3}$ for every $e \in K_n$. The proof that  $\pr{0<\tau=\tau_{Y_e}^{-}<m_{*}}\le n^{-3}$ is essentially identical.
		Our approach is similar to that of the previous section except we now use Freedman's inequality (Lemma~\ref{lem:Freedman}). 
		
		From now on, we fix $e \in K_n$ and suppose that the event $ \{0<\tau=\tau_{Y_e}^{+}<m_{*}\}$ occurs. 
		As ${\tau >0}$, $G_0$ satisfies all items (1)--(4) in Proposition~\ref{prop:I}.
		In particular, $Y_e(G_0) = (1\pm g_Y(0)/2)\ty(0)$, and hence there exists a step $i \in \{1,\ldots,\tau\}$ for which
		$Y_e(G_{i-1}) \le (1 + g_Y(i-1)/2)\ty(i-1)$ and $Y(G_{m}) > (1 + g_Y(m)/2)\ty(m)$ for all $m \in \{i,\ldots,\tau\}$.

		As $Y_e(G_i)$ and $Y_e(G_{i-1})$ are ``very close'' to each other, we have an upper bound on $Y_e(G_i)$ which is better than $Y_e(G_i) \le (1 + g_Y(i))\ty(i)$.
		Note that
		\begin{align}\label{eq:boundonYi}
			Y_e(G_{i}) \le Y_e(G_{i-1}) < \left(1 + \dfrac{g_Y(i-1)}{2}\right)\ty(i-1) \le  \left(1 + \dfrac{3g_Y(i)}{4}\right)\ty(i).
		\end{align}
		The last inequality follows from the fact that $g_Y(i) = g_Y(i-1)\big(1+n^{-2+o(1)}\big)$ and that $\ty(i)=\ty(i-1)\big(1-n^{-2+o(1)}\big)$.
		In particular, $X_i' \le -g_Y(i)\ty(i)/4$.
		As $X_{m_{*}}'=X_{\tau}'>0$, we obtain $X_{m_{*}}'-X_i' \ge g_Y(i)\ty(i)/4$.
		For each $i \in \{1,\ldots,m_{*}\}$, define
		\begin{align*}
			A_i' :=	\left \{ X_{m_{*}}'-X_i'\ge\dfrac{g_Y(i)\ty(i)}{4}\right \} \cap \left \{Y_e(G_i) > \left(1 + \dfrac{g_Y(i)}{2}\right)\ty(i) \text{ and } \psi_i' = m_{*}\right \}.
		\end{align*}
		From all the discussion above, it follows that
		\begin{align}\label{eq:union-A'}
			\{0<\tau=\tau_Y^{+}<m_{*}\} \se \bigcup  _{i=1}^{m_{*}}	A_i'.
		\end{align}	
		
		We now fix some $i \in [m_{*}]$ and bound the probability of the event $A_i'$.
		We encode this event as a deviation event of a supermartingale.
		As the process freezes at the stopping time $\tau$, we have $\psi_i' = m_{*} $ on the event $A_i'$. By Lemmas~\ref{lem:Ysm} and~\ref{lem:Ychange}, $(X'_m)_{m\ge i}$ is a supermartingale whose increments have quadratic variation bounded by
		\begin{align*}
			\Ex{(X'_{m+1}-X'_{m})^2|G_{m}}\, &\le\, \|X'_{m+1}-X'_{m}\|_{\infty}\Ex{|X'_{m+1}-X'_{m}| \, \big|G_{m}}\\
			& =\, \frac{O(1)k^6n^{\delta}\max\{1,\Ex{Y_S(G_0)}\}\ty(i)}{pn^2},
		\end{align*}
		for every $m\ge i$, where $S \se [n]$ is an arbitrary set of size three.
		Summing over $m_{*}-i\le m_{*}\le pn^2$ terms we obtain
		\begin{align}\label{eq:V} \sum_{m=i}^{m_{*}-1}\Ex{(X'_{m+1}-X'_{m})^2|G_{m}}\, =\, O(1)k^6n^{\delta}\max\{1,\Ex{Y_S(G_0)}\}\ty(i).
		\end{align}

		Set $W = k^6n^{\delta}\max\{1,\Ex{Y_S(G_0)}\}\ty(i)$.
		By Lemma~\ref{lem:Ychange}, we have $|X_{m+1}'-X_{m}'| \cdot g_Y(i) \ty(i) = o\big(W\big)$ for every $m\ge i$, and hence by Freedman's inequality we have
		\begin{align}\label{eq:FreedZ}
			\Pr\big( A_i'\big) \le
			\exp \left (-\Omega \left (\dfrac{(g_Y(i) \ty(i))^2}{W}\right) \right),
		\end{align}
		It remains to compare the expressions $(g_Y(i) \ty(i))^2$ and $W$.
		As $g_Y(i)\ge g_Y(0)>n^{-\delta}$, we have
		\begin{align}\label{eq:first-ratio-freedman}
			\dfrac{(g_Y(i) \ty(i))^2}{W}= \Omega \left (  \dfrac{\ty(i)}{k^6n^{3\delta}\max\{1,\Ex{Y_S(G_0)}\}}\right).
		\end{align}

		Now, we claim that  $\ty(i) \ge n^{4\delta} \max\{1,\Ex{Y_S(G_0)}\}$.
		Indeed, as  $i \le m_{*}$, we have
		\begin{align}
			\ty(i) \, = \,\Omega \left ( \dfrac{k^2\tq(0)}{n^{2+4\delta}}\right)
			\, = \, \Omega \left ( n^{\gamma-2-4\delta+o(1)}\right).\label{eq:2nd-ratio-freedman}
		\end{align}
		As $\max\{1,\Ex{Y_S(G_0)}\}=n^{\max\{0,\gamma-3\}+o(1)}$ and $\delta=\min\{\gamma-2,1\}/10$, it follows from~\eqref{eq:2nd-ratio-freedman} that $\ty(i) \ge n^{4\delta} \max\{1,\Ex{Y_S(G_0)}\}$.
		By combining this with~\eqref{eq:first-ratio-freedman} we conclude that
		\begin{align}\label{eq:boundFreed}
			\dfrac{(g_Y(i) \ty(i))^2}{W}= \Omega \left (k^{-6}n^{\delta}\right).
		\end{align}
		From~\eqref{eq:FreedZ} and~\eqref{eq:boundFreed}, it follows that
		\begin{align}\label{eq:bound-Ai'}
			\Pr\big( A_i'\big) =
			\exp \left(-\Omega\left(k^{-6}n^{\delta}\right)\right) \le n^{-5}.
		\end{align}
		
		By the union bound over $0<i<m_{*}$ and $e \in E(K_n)$, it follows from~\eqref{eq:union-A'} and~\eqref{eq:bound-Ai'} that
		\begin{align*}
			\Pr \left( \bigcup_{e \in E(K_n)} \big\{ 0< \tau=\tau_{Y_e}^{+}<m_{*} \big\} \right) \le n^{-1}.
		\end{align*}
		This finishes the proof.
		
	\end{proof}

	To prove Lemmas~\ref{lem:Ysm} and~\ref{lem:Ychange}, we first need a bound on the increments which compose the deterministic part of $(X'_m)_{m\ge 0}$.
	To do so, it is useful to define the following error functions:
	\begin{align}\label{errorE1E2}
		E_1' := \dfrac{k^6\ty(m)}{e(G_{m+1})^2} \qquad \text{and} \qquad E_2' := \dfrac{k^8g_Y(m)\ty(m)}{e(G_{m+1})^2}.
	\end{align}
	We also set 
	\begin{align}\label{eq:R_m}
		R_m' := g_Y(m+1)\ty(m+1)-g_Y(m)\ty(m).
	\end{align}
	
	\begin{claim}\label{claim:deltaty}
		For all $m \in \{0,\ldots,m_{*}-1\}$, we have
		\begin{align*}
			\ty(m+1)-\ty(m)=-\binom{k}{2} \left(\binom{k}{2}-1\right)  \dfrac{\ty(m)}{e(G_{m})}\, \pm \, E_1'.
		\end{align*}
		and
		\begin{align*}
			R'_m = g_Y(m) \binom{k}{2} \left(\binom{k}{2}-1\right) \dfrac{\ty(m)}{e(G_m)}\, \pm \, E_2'.
		\end{align*}
	\end{claim}
	\begin{proof}
		First, let us analyse the one-step difference $\ty(m+1)- \ty(m)$. Recall from~\eqref{def:tq} and~\eqref{def:ty} that
		\begin{align*}
			\tq(m+1)=\left(1-\dfrac{\binom{k}{2}^2}{e(G_m)}\right)\tq(m)
			\qquad \text{and} \qquad
			\ty(m)=\binom{k}{2}\dfrac{\tq(m)}{e(G_m)}.
		\end{align*}
		From this, it follows that
		\begin{align}\label{eq:deltaty}
			\ty(m+1)- \ty(m) \, & = \, \binom{k}{2}\left(\dfrac{\tq(m+1)}{e(G_{m+1})}-\dfrac{\tq(m)}{e(G_m)}\right) \nonumber\\
			& =\, \binom{k}{2}\tq(m)\left(\dfrac{1}{e(G_{m+1})}\left(1-\dfrac{\binom{k}{2}^2}{e(G_m)}\right)-\dfrac{1}{e(G_m)}\right).
		\end{align}
		As $e(G_{m+1})=e(G_m)-\binom{k}{2}$, we have $e(G_{m+1})^{-1}-e(G_m)^{-1}=\binom{k}{2}e(G_{m})^{-1}e(G_{m+1})^{-1}$.
		Thus, it follows from~\eqref{eq:deltaty} that
		\begin{align*}
			\ty(m+1)- \ty(m) \,  = \, -\binom{k}{2}^2 \left(\binom{k}{2}-1\right) \dfrac{\tq(m)}{e(G_{m+1})e(G_m)} \, = \, -\binom{k}{2} \left(\binom{k}{2}-1\right)  \dfrac{\ty(m)}{e(G_{m+1})}.
		\end{align*}
		As $e(G_{m+1})^{-1}=e(G_{m})^{-1}\pm k^2e(G_{m+1})^{-2}$ , we obtain 
		\begin{align}\label{eq:ty-difference}
			\ty(m+1)- \ty(m) = -\binom{k}{2} \left(\binom{k}{2}-1\right)  \dfrac{\ty(m)}{e(G_{m})} \, \pm \, E_1'.
		\end{align}
		This proves the first part of the claim.
		
		Now, let us analyse $R_m'$.
		Note that
		\begin{align}\label{eq:derivative}
			R_m' = g_Y(m)(\ty(m+1)-\ty(m))+ (g_Y(m+1)-g_Y(m)) \ty(m+1).
		\end{align}
		We shall analyse each of the terms in the right-hand side of~\eqref{eq:derivative} separately.
		From~\eqref{eq:ty-difference}, it follows that
		\begin{align}\label{eq:g_Ydeltaty}
			g_Y(m) (\ty(m+1)- \ty(m)) = -\binom{k}{2} \left(\binom{k}{2}-1\right)  \dfrac{g_Y(m)\ty(m)}{e(G_{m})}\, \pm \,
			\dfrac{E_2'}{2}.
		\end{align}
		Now it remains to bound the term $(g_Y(m+1)- g_Y(m))\ty(m+1)$.
		From the definition of $g_Y$ in~\eqref{eq:gY}, we have
		\begin{align}\label{eq:deltag_Y}
			g_Y(m+1)- g_Y(m) = 2\binom{k}{2}^2\dfrac{g_Y(m)}{e(G_m)}.
		\end{align}
		Moreover, from the definitions of $\tq$ and $\ty$ (see~\eqref{def:tq} and~\eqref{def:ty} in Section~\ref{sec:over}) it follows that
		\begin{align}\label{eq:tym+1}
			\ty(m+1)= \binom{k}{2}\left(1-\dfrac{\binom{k}{2}^2}{e(G_{m})}\right) \dfrac{\tq(m)}{e(G_{m+1})}.
		\end{align}
		By combining~\eqref{eq:deltag_Y} and~\eqref{eq:tym+1}, we obtain
		\begin{align*}
			(g_Y(m+1)-g_Y(m))\ty(m+1) = 2\binom{k}{2}^2 \left(1-\dfrac{\binom{k}{2}^2}{e(G_{m})}\right) \dfrac{g_Y(m)\ty(m)}{e(G_{m+1})}.
		\end{align*}
		Again because $e(G_{m+1})^{-1}=e(G_{m})^{-1}\pm k^2e(G_{m+1})^{-2}$, we obtain
		\begin{align}\label{eq:deltagyty}
			(g_Y(m+1)-g_Y(m)) \ty(m+1) = 2\binom{k}{2}^2 \dfrac{(g_Y\ty)(m)}{e(G_{m})
			} \, \pm \, \dfrac{E_2'}{2}.
		\end{align}
		The claim follows from combining equations~\eqref{eq:derivative},~\eqref{eq:g_Ydeltaty} and \eqref{eq:deltagyty}.
	\end{proof}

	For each ordered pair of distinct edges $(e,f) \in E(K_n)^2$, 
	define $Y_{e,f}(G_m)$ to be the number of copies of $k$-cliques in $G_m \cup \{e\}$ containing $e$ and $f$ (observe that order matters).
	Our next proposition expresses $\Ex{ Y_e(G_{m+1})-Y_e(G_{m})|G_m}$ in terms of these variables.
	Define
	\begin{align}\label{errorE3}
		E_3' := \dfrac{2\ty(m)k^3n^{\delta}}{\tq(m)}\cdot\max\{1,\Ex{Y_S(G_0)}\}.
	\end{align}
	
	\begin{prop}\label{prop:changeY} Suppose that $\tau >0$. Then, for all $m \in \{0,\ldots,\tau-1\}$, we have
		\begin{align*}
			\Ex{Y_{e}(G_{m+1})-Y_{e}(G_{m})\big|G_{m}}\, =\, -\dfrac{1}{Q(G_{m})}{\sum}_{f\in E(G_{m})\setminus\{e\}}Y_{e,f}(G_{m})Y_{f}(G_{m})\, \pm\, E_3'.
		\end{align*} 
	\end{prop}
	
	\begin{proof}
		In order to analyse the one-step change $Y_{e}(G_{m+1})-Y_{e}(G_{m})$, 
		we define $P_e(m)$ to be the number of pairs $(K^{1},K^{2})$ satisfying
		\begin{enumerate}
			\item [$1.$] $K^{1}$ is a $k$-clique in $G_m$;
			\item [$2.$] $K^{2}$ is a $k$-clique in $G_m \cup \{e\}$ containing $e$ and at least one edge of $K^{1}\setminus \{e\}$.
		\end{enumerate}
		At step $m+1$ of the process, 
		recall that a $k$-clique is uniformly chosen among those in $G_{m}$.
		This implies that the conditional expectation $\Ex{Y_{e}(G_{m+1})-Y_{e}(G_{m}) |G_m}$ is precisely
		\begin{align}\label{eq:ExY}
			\Ex{Y_{e}(G_{m+1})-Y_{e}(G_{m})|G_m}= -\dfrac{P_e(m)}{Q(G_m)}.
		\end{align}
		
		Let us give an upper bound on $P_e(m)$. From items $1$ and $2$ we can easily see that
		\begin{align}\label{ineq:Pm}
			P_e(m)\le \sum \limits_{K} \sum \limits_{f \in K \setminus \{e\}} Y_{e,f}(G_m) = \sum \limits_{f \in E(G_m)\setminus \{e\}} Y_{e,f}(G_{m}) Y_f(G_m),
		\end{align}
		where the first summation is over all $k$-cliques in $G_m$.
		We might not have an equality above due to a possible slight overcount.
		Those pairs of cliques $(K^1,K^2)$ where $K^2$ contains at least $3$ vertices in $K^1$ are counted more than once in the sums above. 
		But, we can easily see that the overcount does not exceed 
		\begin{align}\label{eq:overcount}
			Y_e(G_{m})\binom{k}{3}\max\{Y_{S}(G_{m}):|S|=3\}.
		\end{align}
		As $m<\tau$, we have $Y_e(G_m)\le 2\ty(m)$.
		Moreover, since $\tau >0$, $G_0$ satisfies all items (1)--(4) in Proposition~\ref{prop:I}.
		In particular, we have $Y_{S}(G_{m})\le Y_{S}(G_{0}) \le n^{\delta}\max\{1,\Ex{Y_S(G_0)}\}$ for all $S \se [n]$ such that $|S|=3$.
		Thus, the quantity in~\eqref{eq:overcount} is upper bounded by $$\ty(m)k^3n^{\delta}\max\{1,\Ex{Y_S(G_0)}\}$$
        for all $S \se [n]$ such that $|S|=3$.
		From~\eqref{eq:ExY} and~\eqref{ineq:Pm}, it follows that
		\begin{align*}
			\Ex{Y_e(G_{m+1})-Y_e(G_m)\big|G_m} = -\dfrac{1}{Q(G_m)}\sum \limits_{f \in E(G_m)\setminus \{e\}} Y_{e,f}(G_{m})Y_f(G_m)\pm E_3'.
		\end{align*}
		This completes the proof.
	\end{proof}
	
	We are now ready to prove Lemmas~\ref{lem:Ysm} and~\ref{lem:Ychange}.

	\begin{proof}[Proof of Lemma~\ref{lem:Ysm}]
		We observe that the result holds trivially if either $\psi'_i = i$ or $\tau = 0$. We therefore assume that $\psi_i' > i$ and $\tau > 0$ throughout the proof.
		
		Fix $m \in \{i,\ldots, \psi_i'-1\}$. 
		As the increment $X_{m+1}'-X_{m}'$ is identically $0$ if $m \ge \tau$, we may assume that $m < \tau$ and so the increment $X_{m+1}'-X_{m}'$ is given by
		\[
		\Big(Y_e(G_{m+1})\, -\, Y_e(G_{m})\Big)\, -\,  \Big((1+g_Y(m+1))\ty(m+1)\, -\, (1+g_Y(m))\ty(m)\Big).
		\]
		Our aim is to prove that $\Ex{X'_{m+1}-X'_{m}|G_{m}}\le 0$,
		which is equivalent to
		\eq{suffY}
		\mathbb{E}\Big[Y_e(G_{m+1})\, -\, Y_e(G_{m})\Big| G_{m}\Big]\, -\,  \Big((1+g_Y(m+1))\ty(m+1)\, -\, (1+g_Y(m))\ty(m)\Big)\, \le 0\, .
		\eqe
		
		In order to analyse the one-step change $Y_e(G_{m+1})-Y_e(G_{m})$, we first observe that
		\begin{align}\label{eq:sum-yef}
			\sum \limits_{f \in G_m \setminus \{e\}} Y_{e,f}(G_m) =
			\sum \limits_{f \in G_m \setminus \{e\}} \sum \limits_{ K \ni e,f } 1_{\{K \subseteq G_m \cup \{e\}\}} = \left(\binom{k}{2}-1\right)Y_e(G_m).
		\end{align}
		Since  $m < \tau$, we have $Y_f(G_m) \ge (1-g_Y(m))\ty(m)$ for all $f \in E(G_m)$.
		Moreover, since $\tau >0$, from Proposition~\ref{prop:changeY} and~\eqref{eq:sum-yef}, it follows that
		\begin{align}\label{eq:conditional-ye-1}
			\Ex{Y_e(G_{m+1})-Y_e(G_m)\big|G_{m}}\, \le \, -(1-g_Y(m))\left(\binom{k}{2}-1\right)\dfrac{\ty(m)Y_e(G_m)}{Q(G_{m})}
			\, \pm\, E_3'.
		\end{align}
		As $Q(G_{m})\le (1+g_Q(m))\tq(m)$ and $\ty(m)=\binom{k}{2}\tq(m)/e(G_m)$, it follows from~\eqref{eq:conditional-ye-1} that $\Ex{Y_e(G_{m+1})-Y_e(G_m)\big|G_{m}}$ is at most
		\begin{align}\label{eq:expdeltaY}
			-(1-g_Y(m)-g_Q(m))\left(\binom{k}{2}^2-\binom{k}{2}\right)\dfrac{Y_e(G_m)}{e(G_{m})}
			\, \pm\, E_3'.
		\end{align}

		As  $i \le m < \psi'_i$, we also have
		$Y_e(G_{m})\ge (1+g_Y(m)/2)\ty(m)$.
		Now, we would like to replace $Y_e(G_m)$ in the right-hand side of~\eqref{eq:expdeltaY} by this lower bound.
		Before that, let us bound the cumulative error coming from $g_Q$ and $g_Y$. 
		As $g_Q(m)\le g_Y(m)/5$, we have
		\begin{align}\label{eq:errorgqtimes}
			\big(1-g_Q(m)-g_Y(m)\big)\big(1+g_{Y}(m)/2\big)&\ge \big(1-6g_Y(m)/5\big)\big(1+g_{Y}(m)/2\big)\nonumber\\
			& = 1-7g_Y(m)/10-3g_Y(m)^2/5\nonumber\\
			& \ge 1-3g_Y(m)/4.
		\end{align}
		Finally, by replacing $Y_e(G_m)$ by $(1+g_Y(m)/2)\ty(m)$ in the right-hand side of~\eqref{eq:expdeltaY} and using~\eqref{eq:errorgqtimes}, we obtain
		\begin{align}\label{eq:deltaexY}
			\Ex{Y_e(G_{m+1})-Y_e(G_m)\big|G_{m}}\, \le \, -(1-3g_Y(m)/4)\left(\binom{k}{2}^2-\binom{k}{2}\right)\dfrac{\ty(m)}{e(G_{m})}
			\, \pm\, E_3'.
		\end{align}

		We remark that this is why we need to use the bound $Y_e(G_m)\ge(1+g_Y(m)/2)\ty(m)$.
		If we only use the crude bound $Y_e(G_m)\ge(1-g_Y(m))\ty(m)$, the coefficient of $g_Y(m)$ would be bigger than $1$.
		This would not be enough to show that $(X'_m: i \le m \le \psi'_i)$ is a supermartingale.
		
		By combining Claim~\ref{claim:deltaty} with~\eqref{eq:deltaexY}, we obtain
		\begin{align}\label{expdeltaZ}
			\Ex{X_{m+1}'-X_m'\big|G_{m}}\, \le \, -\dfrac{g_Y(m)}{4}\left(\binom{k}{2}^2+\binom{k}{2}\right)\dfrac{\ty(m)}{e(G_{m})}
			\, \pm\, E',
		\end{align}
		where $E' = E_1'+E_2'+E_3'$.
		Now we only need to bound the sum of the errors $E':=E_1'+E_2'+E_3'$.
		Recall the definition of $E_1'$ and $E_2'$ in~\eqref{errorE1E2} and the definition of $E_3'$ in~\eqref{errorE3}.

	\begin{claim}\label{eq:boundonE'}
		$E' \ll g_Y(m)\ty(m)/e(G_m).$
	\end{claim}

	\begin{proof}
		From the definitions of the error functions $E_1'$ and $E_2'$ in~\eqref{errorE1E2}, we can easily see that
		\begin{align}\label{eq:e1ande2prime}
			E_1'+E_2' \ll \dfrac{g_Y(m)\ty(m)}{e(G_m)}.
		\end{align}

		For $E_3'$, recall that
		\begin{align}\label{recap:E3}
			E_3' = \dfrac{2\ty(m)k^3n^{\delta}}{\tq(m)}\cdot\max\{1,\Ex{Y_S(G_0)}\},
		\end{align}
		where $S$ is a set of size three.
		As $g_Q(m)\tq(m) \ge n^{-\delta}\tq(0)$ and $\delta = \min\{1,\gamma-2\}/10$, we have
		\begin{align}\label{ineq:boundE3}
			k^3n^{\delta}\max\{1,\Ex{Y_S(G_0)}\} \le n^{\delta+\max\{0,\gamma-3\}+o(1)} \ll n^{-\delta+\gamma-2+o(1)} \le \frac{g_Q(m)\tq(m)}{4e(G_m)}.
		\end{align}
		By multiplying the inequalities in~\eqref{ineq:boundE3} by $\ty(m)/\tq(m)$ and using~\eqref{recap:E3}, we obtain
		\begin{align}\label{eq:e3prime}
			E_3' \ll \dfrac{g_Q(m)\ty(m)}{e(G_m)}.
		\end{align}
		The claim follows by combining~\eqref{eq:e1ande2prime} with~\eqref{eq:e3prime}.
	\end{proof}

		It follows from Claim~\ref{eq:boundonE'} that the error terms $E'$ is much smaller than the leading term in~\eqref{expdeltaZ}. Therefore, we conclude that~\eqr{suffY} holds, and hence $(X'_m: i \le m \le \psi'_i)$ is a supermartingale.
	\end{proof}

	Now we proceed to the proof of Lemma~\ref{lem:Ychange}. Recall from~\eqref{eq:R_m} that
	\begin{align*}
		R_m' := g_Y(m+1)\ty(m+1)-g_Y(m)\ty(m).
	\end{align*}
	
	\begin{proof}[Proof of Lemma~\ref{lem:Ychange}]
		We observe that the result holds trivially if $m \ge \tau$ or $\tau = 0$, since the increment $X'_{m+1}-X'_m$ is identically $0$. We therefore assume that $\tau > m \ge 0$ throughout the proof.

		We begin with part $(i)$. For every $m< \tau$, it suffices to show that
		$|Y_e(G_{m+1})-Y_e(G_{m})|$ and $| \ty(m+1)-\ty(m)+R_m'|$ are bounded by $k^2n^{\delta}\max\{1,\Ex{Y_S(G_0)}\}$.
		Let $K$ denote the $k$-clique chosen at step $m$. Then, we have
		\begin{align}\label{ineqdeltaYeGm}
			|Y_e(G_{m+1})-Y_e(G_{m})| \,  \le \, \sum_{f\in K\setminus\{e\}} Y_{e,f}(G_{m}) 
			\, \le \, \binom{k}{2}\max\big\{Y_{S}(G_0):|S|=3\big\}.
		\end{align}
		Since $\tau >0$, $G_0$ satisfies all items (1)--(4) in Proposition~\ref{prop:I}, and hence we have $Y_{S}(G_{m})\le Y_{S}(G_{0}) \le n^{\delta}\max\{1,\Ex{Y_S(G_0)}\}$.
		This combined with~\eqref{ineqdeltaYeGm} provides the required upper bound for $|Y_e(G_{m+1})-Y_e(G_{m})|$.
		
		Now we bound $| \ty(m+1)-\ty(m)+R_m'|$.
		By Claims~\ref{claim:deltaty} and~\ref{eq:boundonE'}, we have
		\begin{align}\label{eq:boundY+gY}
			|\ty(m+1)-\ty(m)+R_m'| \le \dfrac{4k^4\ty(m)}{e(G_m)}.
		\end{align}
		As $m \le m_{*}$, we have $e(G_m)\ge pn^2/4$ and $\ty(m) \le \ty(0) \le 2 \Ex{Y_e(G_0)}$, which imply that
		\begin{align*}
			|\ty(m+1)-\ty(m)+R_m'| & \le \dfrac{2^5k^4\Ex{Y_e(G_0)}}{pn^2}\, \le \, k^2n^{\delta}\max\{1,\Ex{Y_S(G_0)}\}.
		\end{align*}

		Now let us prove part $(ii)$.
		By using~\eqref{eq:boundY+gY}, it suffices to show that
		\begin{align*}
			\Ex{| Y_e(G_{m+1})-Y_e(G_m)|\big|G_m}=  O \left (\dfrac{k^4\ty(m)}{pn^2}\right).
		\end{align*}
		Recall that, by Proposition~\ref{prop:changeY}, we have
		\begin{align}\label{eq:expYeGm}
			\Ex{|Y_{e}(G_{m+1})-Y_{e}(G_{m})|\big|G_{m}}\, =\, \dfrac{1}{Q(G_{m})}{\sum}_{f\in E(G_{m})\setminus\{e\}}Y_{e,f}(G_{m})Y_{f}(G_{m})\, \pm\, E_3',
		\end{align} 
		where the error term $E_3'$ is given in~\eqref{errorE3}.
		In the proof of Lemma~\ref{lem:Ysm}, we have seen in~\eqref {eq:e3prime} that  $E_3' \ll g_Q(m)\ty(m)/(pn^2)$. As $g_Q(m) \ll 1 \le k^4$, we then have 
        $E_3' \le k^4\ty(m)/(pn^2)$.
		Thus, it remains to bound the leading term in~\eqref{eq:expYeGm}.
		
		Observe that
		\begin{align*}
			\sum \limits_{f \in E(G_m)\setminus \{e\}} Y_{e,f}(G_m) = 
			\left( \binom{k}{2}-1\right)Y_e(G_m).
		\end{align*}
		Moreover, on the event $\{m< \tau\}$ we have $Y_{f}(G_{m})\le (1+g_Y(m))\ty(m)$ for all $f\in E(K_n)$.
		These imply that the leading term of $\Ex{|Y_{e}(G_{m+1})- Y_{e}(G_{m})|\big|G_{m}}$ in~\eqref{eq:expYeGm} is upper bounded by
		\begin{align*} 
			(1+g_Y(m)) \left( \binom{k}{2}-1\right)
			\dfrac{Y_e(G_m)\ty(m)}{Q(G_{m})}.
		\end{align*} 
		As $Q(G_{m})\ge (1-g_Q(m))\tq(m)$ and $Y_{e}(G_{m})\le (1+g_Y(m))\ty(m)$ for every $m<\tau$,
		by combining~\eqref{eq:expYeGm} with the bound on the leading term we obtain
		\begin{align*}
			\Ex{|Y_{e}(G_{m+1})- Y_{e}(G_{m})|\big|G_{m}}\, & \le \, 
			\dfrac{4 k^2(\ty(m))^2}{\tq(m)}
			\, \pm\, E_3'\, \le \, \dfrac{8 k^4\ty(m)}{pn^2}.
		\end{align*} 
		This concludes our proof.
	\end{proof}

	\section{Upper bounds}\label{sec:upper}

	In this section we prove Theorem~\ref{thm:upper},
	which gives an upper bound on the size of the maximum $k$-clique packing in $G(n,p)$.  We first recall some notation.  We denote by $\NN(n,k,t)$ the number of sets of $t$ edge-disjoint $k$-cliques in $K_n$, and define $\zeta(n,k,t)$ to be such that
	\[\NN(n,k,t)\, =\, \dfrac{1}{t!}\binom{n}{k}^t \zeta(n,k,t).\]
	We have
	\[t_0:= \dfrac{5(\gamma-2)}{1-p} \cdot \dfrac{pn^2 \log{n}}{k^4}\]
	and define $\beta=\beta(n,k)$ to be maximal such that the inequality
	\eq{betar}
	\zeta(n,k,t)\, \le \, \exp\left(\frac{-\beta t^2k^4}{n^2}\right) 
	\eqe
	holds for all $t\le t_0$.    For the reader's convenience, we restate Theorem~\ref{thm:upper}.

	\upper*

	As in~\cite{AK}, our proof uses the first moment method, but with one small tweak -- we work in $G(n,m)$ instead of $G(n,p)$.
	As usual, $G(n,m)$ denotes the random graph selected uniformly at random from graphs with $n$ vertices and $m$ edges.   
	Let us briefly explain how these two models are related.  One may reveal a random graph $G\sim G(n,p)$ by first revealing the number of edges $e(G)=m$, and then revealing $G$.  The conditioned random graph (condition on $m$, the number of edges) is then distributed as $G(n,m)$.
	One may also relate the probabilities of an event in the two models.
	Let $\Pr_p$ be the probability measure associated with $G(n,p)$ and $\Pr_m$ the probability measure associated with $G(n,m)$.
	For an event $E$, we have
	\[
	\Pr_p(E)\, =\, \sum_{m=0}^{N}\pr{Bin(N,p)=m}\Pr_m(E),
	\]
	where $N:=\binom{n}{2}$.
	From this it follows that if $E$ is an increasing event, then for every $m_{+}\in \{0,\dots ,N\}$ we have
	\begin{align}\label{eq:GnpGnm}
		\Pr_p(E)\, \le\, \Pr_{m_+}(E)\, +\, \pr{Bin(N,p)>m_+}.
	\end{align}
	
	The elementary observation in~\eqref{eq:GnpGnm} turns out to be surprisingly useful.  
	For $t \in \mathbb{N}$,
	let $E_t$ to be the event that the graph contains a family of $t$ edge-disjoint $k$-cliques.
	Now, define $m_+:= \lfloor pN+(pN)^{3/4} \rfloor$.  By Chernoff's inequality, we have $\pr{Bin(N,p)>m_+}\le \exp(-\Omega(\sqrt{pn^2})) = \exp(-n^{1+o(1)})$.
	As $E_t$ is an increasing event, by~\eqref{eq:GnpGnm} it follows that $\Pr_p(E_t)\, \le\, \Pr_{m_+}(E_t)\, +\, \exp(-n^{1+o(1)})$.

	Let $X_t$ be the variable which counts the number of collections of $t$ edge-disjoint $k$-cliques.
	By Markov's inequality, we have $\Pr_{m_+}(E_t)\le  \Exm{X_t}$ and,
	by putting all pieces together, we obtain
	\begin{align}\label{eq:plee}
		\Pr_p(E_t)\, &\le\, \Exm{X_t}\, +\, \exp(-n^{1+o(1)}).
	\end{align}
	Thus, in order to prove our theorem, we must show that $\Exm{X_t}=o(1)$ when we take 
	\begin{align}\label{eq:valueoft}
		t:=\dfrac{(\gamma -2)(4+\eps)}{1+(4\beta-1)p} \cdot \dfrac{pn^2}{k^4}\log n.
	\end{align}
	
	From now on, our goal is to upper bound $\Exm{X_t}$.
	The first step is to fix a family $\K= \{K^1,\dots, K^t\}$ of $t$ edge-disjoint $k$-cliques and bound the probability that their union of edges $E_{\K} := E(K^1) \cup \dots \cup E(K^t)$ is in $G(n,m)$.
	For simplicity, let us write $(n)_a$ for the falling factorial $n(n-1)\dots (n-a+1)$.
	The probability that $E_{\K}\se G(n,m)$ is precisely
		\begin{align}\label{eq:EKinGnM-1}
		\pr{E_{\K}\se G}\, &=\, \dfrac{(m_+)_{t\binom{k}{2}}}{(N)_{t\binom{k}{2}}}\nonumber \\ 
		&=\, \left(\dfrac{m_{+}}{N}\right)^{t\binom{k}{2}} \cdot  \dfrac{\prod_{i=0}^{t\binom{k}{2}-1}(1-i/m_{+})}{\prod_{i=0}^{t\binom{k}{2}-1}(1-i/N)}\nonumber \\
		&=\, \left(\dfrac{m_{+}}{N}\right)^{t\binom{k}{2}} \exp\left(\sum_{i=0}^{t\binom{k}{2}-1}\big(\log(1-i/m_{+})-\log(1-i/N))\right)\, .
	\end{align}

	As $x+\log(1-x)$ is a decreasing function of $x$, for $x \in (0,1)$, we have
\begin{align*}
	\log(1-i/m_{+})-\log(1-i/N)\, \le\, \frac{i}{N}\, -\, \frac{i}{m_{+}}\, .
\end{align*}
Using this, and the fact that $m_+N^{-1}\le p(1+O((pN)^{-1/4}))=p\exp(O((pN)^{-1/4}))$, we obtain
\begin{align*}
	\pr{E_{\K}\se G} \, &=\, p^{t\binom{k}{2}}
	\, \exp\left(O\left(\dfrac{tk^2}{(pN)^{1/4}}\right)\, -\sum_{i=0}^{t\binom{k}{2}-1}\left(\frac{i}{m_{+}}-\frac{i}{N}\right)\right)\\ 
	& = \, p^{t\binom{k}{2}}
	\exp\left(
	\,-\, \binom{t\binom{k}{2}}{2}\cdot \left(\frac{1}{m_{+}}-\frac{1}{N}\right)
	\,+\, O\left(\dfrac{tk^2}{(pN)^{1/4}}\right)\right).
\end{align*}

As $t= n^{2+o(1)}$,	we have
\[
tk^2/(pN)^{1/4} \, =\,  n^{3/2+o(1)}\, =\, o\left(\frac{t^2k^4}{pN}\right)\, .
\]
We may also use that 
\[
\dfrac{1}{m_{+}} - \dfrac{1}{N} = \dfrac{1-p}{pN} + o\left(\dfrac{1}{pN}\right).\]
It follows that
\begin{align}\label{eq:propEK}
	\pr{E_{\K}\se G} \,\le\,	p^{t\binom{k}{2}}\exp\left(-\dfrac{(1-p)t^2k^4}{8pN}\, +\, o\left(\frac{t^2k^4}{pN}\right)\right).
\end{align}

	We are now ready to bound the expected value $\Ex{X_t}$.  There are $\NN(n,k,t)$ different sets of $t$ edge disjoint $k$-cliques, and so, $\Ex{X_t}$ is equal to $\NN(n,k,t)$ times the probability  $\pr{E_{\K}\se G(n,m)}$.  From~\eqref{eq:propEK}, it follows that
	\begin{align}\label{eq:expXt}
		\Exm{X_t}\, =\, \binom{n}{k}^t (t!)^{-1}\zeta(n,k,t)p^{t\binom{k}{2}}\exp\left(-\frac{(1-p)t^2k^4}{4pn^2}\, +\, o\left(\frac{t^2k^4}{pn^2}\right)\right).
	\end{align}
	Now, we replace
	$n^{\gamma}=\binom{n}{k}p^{\binom{k}{2}}$ and $t! \sim \sqrt{2\pi t}\cdot (t/e)^{t} \ge \exp(t\log{t}\, -\, t)$ above to obtain
	\begin{align}\label{eq:expXt-2}
		\Exm{X_t}\, \le \, \zeta(n,k,t)\exp\left(\gamma t\log n \, + \, t \, - \, t  \log t \, - \, \dfrac{(1-p)t^2k^4}{4pn^2}\, +\, o\left(\frac{t^2k^4}{pn^2}\right)\right).
	\end{align}
	
		We now use~\eqr{betar} and that $t \le t_0$ to bound $\zeta(n,k,t)$, and absorb the $t$ into the $o$ error term.  We have
	\begin{align*}
		\Exm{X_t}\, &\le \, \exp\left(-t\left[\frac{tk^4}{pn^2}\left(\beta p\, +\, \frac{1-p}{4}\, +\, o(1)\right)\, -\, \gamma \log n  \, + \,  \log t \right]\right) \\
		& =\, \exp\left(-t\left[\frac{tk^4}{4pn^2}(1+(4\beta-1)p+o(1))\, -\, \gamma \log n  \, + \,  \log t \right]\right)\, .
	\end{align*}
	Substituting in the value of $t$ from~\eqr{valueoft} and using that $t=n^{2+o(1)}$, we obtain that the expression in the square brackets is at least $\eps(\gamma-2)\log{n}/6$, for all sufficiently large $n$. 
	Therefore,
	 \[
	 \Exm{X_t}\, \le \, \exp\left(-\dfrac{\eps t(\gamma-2)\log{n}}{6} \right)\, =\, \exp(-n^{2+o(1)})\, ,
	 \]
	 as required.

	\section*{Acknowledgments}

We thank the anonymous referees for their careful reading and many helpful suggestions.

	\appendix

	\section{Expected number of near maximum cliques} \label{ap:expected}

	Recall that $E_p(n,k)$ denotes the expect number of $k$-cliques in $G(n,p)$ and $k_0 = k_0(n,p)$ is the least integer for which the expected number of $k$-cliques is less than 1.
	We have the following lemma.

	\begin{lemma}\label{lemma:expectation-k-c}
Let $C\in \mathbb{N}$.	 If $k = k_0 - C$, then $n^{C-1+o(1)} \le E_p(n,k) \le n^{C+o(1)}$.
	\end{lemma}

	\begin{proof}
		If $k = (2+o(1))\log_{\frac{1}{p}}n$, then we have 
		\begin{align*}
			\dfrac{E_p(n,k)}{E_p(n,k-1)} &= \dfrac{\binom{n}{k}}{\binom{n}{k-1}} p^{\binom{k}{2}-\binom{k-1}{2}} = 
			\dfrac{n-k+1}{k}p^{k-1} = n^{-1+o(1)}.
		\end{align*}
		As we have $k_0 = (2+o(1))\log_{\frac{1}{p}}n$ (see~\cite{BE,Matula70,Matula72}),
		it follows that 
		\begin{align}\label{eq:Ep}
			E_p(n,k_0-C) = E_p(n,k_0) \cdot n^{C+o(1)}
		\end{align}
		for every constant $C \in \mathbb{N}_{\ge 1}$.
		Now, note that since $E_{p}(n,k_0) < 1 \le E_{p}(n,k_0-1)$, it follows from~\eqref{eq:Ep} that 
		\begin{align}\label{eq:Epk0}
			n^{-1+o(1)} \le E_p(n,k_0-1) \cdot  n^{-1+o(1)} \le E_p(n,k_0) < 1.
		\end{align}
		The lemma follows by combining~\eqref{eq:Ep} and~\eqref{eq:Epk0}.
	\end{proof}

	\section{Proof of Lemma~\ref{lem:onestep}}\label{ap:one-step}
	
	For the reader's convenience, we restate Lemma~\ref{lem:onestep}, which remains to be proved. 
	
	\onestep*

	An important role in our analysis is played by the quantities $\xi_i$, which are defined as
	\begin{align}\label{eq:xi}
		\xi_i \, := \, \frac{n^{k-i}}{(k-i)!}p^{\binom{k}{2}-\binom{i}{2}} \, \ge \, \binom{n}{k-i}p^{\binom{k}{2}-\binom{i}{2}}
	\end{align}
	for each $i=s,\dots,k$ (recall that $s$ denotes the size of $|S|$).   
	Note that $\xi_s^{-1} \cdot \Ex{Y_S(G_0)}\to 1$ as $n$ tends to infinity, and $\xi_k=1$.   Our next lemma just collects some other properties of the sequence $(\xi_i)_{i=s}^k$.
	
	\begin{lemma}\label{lemma:xis} There exists $n_0 \in \mathbb{N}$ such that the following holds for all $n \ge n_0$:
		\begin{enumerate}
			\item[$(1)$] \label{smallxi} $\xi_i\, \le\, n^{-(i-s)/2}\xi_s$ for all $i=s,\dots ,\lfloor k/8 \rfloor$;
			\vspace*{0.1cm}
			\item[$(2)$]  \label{midxi2} $\xi_b\,\le \, \max\{\xi_a,\xi_{c}\}$ for all $s\le a \le b \le c \le  k-D$.
			\vspace*{0.1cm}
			\item[$(3)$]  \label{largexi2} $\xi_i\, \le\, n^{-(k-i)/8}\max\{\xi_s,\xi_{k}\}$ for all $i=\lceil 7k/8 \rceil,\ldots,k.$
		\end{enumerate}
		In particular, 
		\begin{enumerate}
			\item[$(4)$] \label{particular} $\xi_i\, \le\, n^{-D/8}\max\{\xi_s,\xi_{k}\}$ for all $i=D,\dots ,k-D$;
		\end{enumerate}
	\end{lemma}
	
	\begin{proof}
		Note that the ratio of two consecutive terms is given by
		\begin{align}\label{eq:xi-ratio}
			\frac{\xi_{i+1}}{\xi_{i}}\, =\, \frac{k-i}{p^i n}\, .
		\end{align}
		We shall use~\eqref{eq:xi-ratio} to first compare $\xi_i$ with $\xi_s$ for $s \le i \le k/8$.  Fix $s \le i \le k/8$.
		As $k \le 2.1\log_{1/p} n$, we have $i\le( \log_{1/p} n) /3$, and so  $p^in \ge n^{2/3}$.
		It follows that the ratio $\xi_{i+1}/\xi_i$ is at most $n^{-1/2}$.  In particular,
		\eq{smallxi}
		\xi_i\, \le\, n^{-(i-s)/2}\xi_s \qquad\text{for all } i=s,\dots ,  \lfloor k/8 \rfloor.
		\eqe
		This proves the first item of our lemma.

		Now, we compare $\xi_i$ with $\xi_k$ for each $i \in \{\lceil 7k/8 \rceil,\ldots,k\}$.
		If $i \ge \frac{7}{6}\log_{1/p}n$, then $\xi_{i+1} \ge n^{1/6}\xi_i$, and hence $\xi_k \ge n^{(k-i)/6} \xi_i$.
		Otherwise, we have $i < \frac{7}{6}\log_{1/p}n$, and hence $k \le 8i/7 < \frac{4}{3} \log_{1/p}n$. 
		As $k < \frac{4}{3}\log_{1/p} n$, we have $p^{\frac{k-1}{2}}\ge n^{-\frac{2}{3}}$, and hence it follows that
		\begin{align*}
			\xi_s \ge \binom{n}{k-s} p^{\binom{k}{2}-\binom{s}{2}} \ge \dfrac{n^{k-s}}{k^k}p^{\binom{k}{2}}n^{o(1)}\ge \left(\dfrac{np^{\frac{k-1}{2}}}{k}\right)^k n^{-s+o(1)} \ge n^{k/4}
		\end{align*}
		for all sufficiently large $n$.
		Recall that $\xi_k =1$ and every ratio $\xi_{i+1}/\xi_i$ is at least $n^{-1}$. And so, 
		\begin{align}\label{eq:xikandxis}
			\xi_i \le n^{k-i} \le n^{k/8} \le n^{-k/8}\xi_s.
		\end{align}
		This proves the third item of our lemma.
		
		To prove the second item, note that by~\eqref{eq:xi-ratio} we have that the ratio of two consecutive ratios, say $\xi_{i+1}/\xi_i$ over $\xi_i/\xi_{i-1}$ is 
		\[
		\dfrac{\xi_{i+1}}{\xi_i} \cdot \left ( \dfrac{\xi_i}{\xi_{i-1}} \right )^{-1} = \frac{k-i}{k-i+1}p^{-1},
		\vspace*{0.2cm}
		\]
		which is at least $1$ for all $s<i \le k-(1-p)^{-1}$.
		As $\xi_{s+1}/\xi_s <1$ and the ratios are increasing, there must exist an index $j$ such that $(\xi_i)_{i=s}^j$ is non-increasing and $(\xi_i)_{i=j}^{k-D}$ is  non-decreasing (recall that $D >(1-p)^{-1}$). 
		It follows from this that
		\eq{midxi}
		\xi_b\,\le \, \max\{\xi_a,\xi_{c}\} \qquad \text{for all} \, s\le a \le b \le c \le k-D\, .
		\eqe
		This proves the second item, which finishes our proof.
	\end{proof}
	
	\begin{proof}[Proof of Lemma~\ref{lem:onestep}]
		We first prove a general bound on the conditional probabilities. 
		Let $i\in \{s,\dots ,k\}$ and let us consider $k$-cliques $K$ (containing $S$) with $|K\cap U |=i$.  Since $|U|\le kr\le k\ell$, we have that there are at most $\binom{k\ell}{i-s}\binom{n}{k-i}$ ways to select $K$, and the conditional probability that $K$ is in $G_S$ is at most $p^{\binom{k}{2}-\binom{i}{2}}$.  Thus,
		\begin{align}
			\sum_{K:|K \cap U|=i}\mathbb{P}\left(K\, \subseteq\, G_S\, \Big|\, U\, \subseteq\, G_S\, \right)\, &\le\, \binom{k\ell}{i-s}\binom{n}{k-i}p^{\binom{k}{2}-\binom{i}{2}}\nonumber\\ & \le\, \binom{k\ell}{i-s}\xi_i\, .\label{eq:xi-2}
		\end{align}
		Item $(i)$ follows by summing these contributions. 
		By Lemma~\ref{lemma:xis}, we have $\xi_i \le \max \{\xi_s, \xi_k\}$ for every $i \in \{s,\ldots,k\}$ and by~\eqref{eq:xi} we have $\xi_s^{-1} \cdot \Ex{Y_S(G_0)}\to 1$,
		and hence
		\[
		\sum_{i=s}^{k}\binom{k\ell}{i-s}\xi_i\, \le\, \max\{\xi_s,\xi_k\} \sum_{i=s}^{k} \binom{k\ell}{i-s} \,\le\, (2e\ell)^{k} \max\big\{\Ex{Y_S(G_0)},1\big\}\, .
		\]
		
		Item $(ii)$ follows by summing the contributions above only in the range $s\le i\le D$.
		By Lemma~\ref{lemma:xis}, we have $\xi_i \le n^{-(i-s)/2} \xi_s$ for every $i \in \{s,\ldots,D\}$, and hence
		\[
		\sum_{i=s}^{D}\binom{k\ell}{i-s}\xi_i\, \le\,  \sum_{i=s}^{D}(k\ell)^{i-s}\xi_i \,\le\, \sum_{i=s}^{D}\left(\frac{k\ell}{n^{1/2}}\right)^{i-s}\xi_s\, \le\, 2\Ex{Y_S(G_0)}\, .
		\]
		In the last inequality, we used that $k = O( \log_{1/p_{+}}n)$ and that $\xi_s^{-1} \cdot \Ex{Y_S(G_0)}\to 1$.
		
		Item $(iii)$ follows by summing ``middling'' terms.
		By item ($ii$) of Lemma~\ref{lemma:xis}, we have $\xi_i \le \max \{ \xi_{D}, \xi_{k-D}\}$ for all $i \in \{D+1,\ldots,k-D-1\}$.
        Moreover, by items $(i)$ and $(iii)$ of Lemma~\ref{lemma:xis}, we have
        \begin{align*}
            \xi_D \le n^{-\frac{D-s}{2}}\xi_s \le n^{-D/8} \xi_s \qquad \text{and}  \qquad \xi_{k-D} \le n^{-D/8}\max\{\xi_s,\xi_k\},
        \end{align*}
        and hence it follows that $\xi_i \le n^{-D/8}\max\{\xi_s,\xi_k\}$ for all $i \in \{D+1,\ldots,k-D-1\}$. Thus,
		\[
		\sum_{i=D+1}^{k-D-1}\binom{k\ell}{i-s}\xi_i\, \le\, 2^{k\ell} \sum_{i=D+1}^{k-D+1}\xi_i \,\le\, 2^{k\ell} k n^{-D/8}\max\{\xi_s,\xi_k\}\, \le\, n^{-D/16}\max\big\{\Ex{Y_S(G_0)},1\big\}\, .
		\]
		In the last inequality, we used the definition of $D$.
		
		For the final inequality of the lemma we may assume that $(K^1,\dots ,K^r)$ is simple.  
		We shall take into account the maximum value of $j^{*}<r$ such that $(K^1,\dots ,K^{j^*+1})$ is small.
		Note that such a $j^*$ always exists as $(K^1)$ is trivially small. 
		For simplicity, denote $A = K^1 \cup \dots \cup K^{j^*}$ and $B =  K^{j^*+2} \cup \dots \cup K^{r}$. 
		By using that the sequence $(K^1,\dots ,K^t)$ is large for every $t \ge j^{*}+2$, we have 
		\[\left| V(K^{t}) \setminus \bigcup_{i=1}^{t-1} V(K^{i}) \right| \le D\]
		for every $t \ge j^*+2$, and hence it follows that $|B \setminus (A \cup K^{j^*+1})| \le D \ell$.

		Now, let $K^{r+1}$ be a clique in $\C_l$.
		As $(K^1,\dots ,K^{r+1})$ is simple and $(K^1,\dots ,K^{j_{*}+1})$ is small, we have $|K^{r+1} \cap A| \le D$, and hence 
        \begin{align*}
            |K^{r+1} \cap (A \cup B) \setminus K^{j^*+1}| \le |B \setminus (A \cup K^{j^*+1})| + |K^{r+1} \cap A|  \le D(\ell+1).
        \end{align*}
		Moreover, as $K^{r+1}$ intersects $K^1 \cup \dots \cup K^{r} = A \cup B \cup K^{j^{*}+1}$ in at least $k-D$ vertices, we must have $|K^{r+1}\cap K^{j^{*}+1}|\ge k-D(\ell+2)$.
		From this analysis it follows that the number of choices of $K^{r+1}$ with $|K^{r+1}\cap (K^1 \cup \cdots \cup K^r)|=i$ and $i \ge k-D$ is at most 
		\[
		\binom{k}{2\ell D}\binom{k\ell}{2\ell D}\binom{n}{k-i}\, \le\, (k\ell)^{4\ell D}\binom{n}{k-i}\, .
		\]
		Recall that $U = K^1 \cup \cdots \cup K^r$. 
		By Lemma~\ref{lemma:xis} combined with  $\xi_s^{-1} \cdot \Ex{Y_S(G_0)}\to 1$, we obtain
		\[
		\sum_{K^{r+1}\in \C_l}\mathbb{P}\left(K^{r+1} \, \subseteq\, G_S\, \Big|\, U \, \subseteq\, G_S\, \right)\, \le\, (k\ell)^{4\ell D} \sum_{i=k-D}^{k}\xi_i\, \le\, 2(k\ell)^{4\ell D} \max\big\{\Ex{Y_S(G_0)},1\big\}\, .
		\]
		This completes our proof.
	\end{proof}

\end{document}